\documentclass[11pt]{amsart}
\usepackage{amsmath,amssymb}
\usepackage{graphicx}
\usepackage{dsfont}
\usepackage{xcolor}
\usepackage{enumitem}
\usepackage{bm}
\usepackage[margin=1in]{geometry}
\usepackage[colorlinks=true,citecolor=magenta,linkcolor=violet]{hyperref}

\newcommand{\Sf}{\mathbb{S}^2}

\newcommand{\Ker}{\mathrm{Ker\hspace{0.5mm}}}

\newcommand{\Dom}{\mathrm{Dom}}
\newcommand{\norm}[1]{\left\|#1\right\|}
\newcommand{\scalprod}[1]{\left\langle#1\right\rangle}
\newcommand{\R}{\mathbb{R}}
\newcommand{\T}{\mathbb{T}}
\newcommand{\EL}{\mathsf{EL}}
\newcommand{\CH}{\mathsf{CH}}

\newcommand{\dd}{\mathrm{d}}
\newcommand{\mono}{\mathsf{mono}}
\newcommand{\cross}{\mathsf{cross}}
\newcommand{\spaceV}{L^2_v\left(\bm{\mu}^{-\frac{1}{2}}\right)}
\newcommand{\spaceVgamma}{L^2_v\left(\langle v \rangle^{\frac{\gamma}{2}} \bm{\mu}^{-\frac{1}{2}}\right)}

\newtheorem{theorem}{Theorem}
\newtheorem{lemma}{Lemma}
\newtheorem{remark}{Remark}

\makeatletter
\def\blfootnote{\xdef\@thefnmark{}\@footnotetext}
\makeatother

\author[A. Bondesan]{Andrea Bondesan}
\address{Andrea Bondesan \hfill\break
	Department of Mathematical, Physical and Computer Sciences \hfill\break 
    University of Parma \hfill\break
	Parco Area delle Scienze 53/A, 43124 Parma, Italy \vspace*{2mm}}
 
\email{andrea.bondesan@unipr.it, andrea.bondesan@gmail.com \vspace*{5mm}}

\author[B. Q. Tang]{Bao Quoc Tang}
\address{Bao Quoc Tang \hfill\break
	Department of Mathematics and Scientific Computing \hfill\break University of Graz \hfill\break
	Heinrichstrasse 36, 8010 Graz, Austria \vspace*{2mm}}
\email{quoc.tang@uni-graz.at, baotangquoc@gmail.com}

\title[Explicit spectral gap for reactive Boltzmann]{Explicit spectral gap estimates for the\\linearized Boltzmann operator modeling\\reactive gaseous mixtures}

\begin{document}

\maketitle

\vspace*{-0.8cm}
\begin{abstract}
We consider hard-potential cutoff multi-species Boltzmann operators modeling microscopic binary elastic collisions and bimolecular reversible chemical reactions inside a gaseous mixture. We prove that the spectral gap estimate derived for the linearized elastic collision operator can be exploited to deduce an explicit negative upper bound for the Dirichlet form of the linearized reactive Boltzmann operator. This estimate may be used to quantify explicitly the rate of convergence of close-to-equilibrium solutions to the reactive Boltzmann equation toward the global chemical equilibrium of the mixture.
\end{abstract}

\vspace*{0.5cm}

\vspace*{0.5cm}

\tableofcontents

%%%%%%%%%%%%%%%%%%%%%%%%%%%%%%%%%%%%%%%%%%%%%%%%%%%%%%%%%%%%
%%%%%%%%%%%%%%%%%%%%%%%%%%%%%%%%%%%%%%%%%%%%%%%%%%%%%%%%%%%%
%%%%%%%%%%%%%%    SECTION 1: INTRODUCTION    %%%%%%%%%%%%%%%
%%%%%%%%%%%%%%%%%%%%%%%%%%%%%%%%%%%%%%%%%%%%%%%%%%%%%%%%%%%%
%%%%%%%%%%%%%%%%%%%%%%%%%%%%%%%%%%%%%%%%%%%%%%%%%%%%%%%%%%%%

\section{Introduction}

%\blfootnote{\textit{Date:} \today.}

\noindent This paper is concerned with the derivation of an explicit negative upper bound for the Dirichlet form associated with the linearized Boltzmann operator modeling chemical reactions inside a gaseous mixture. This result links directly to the study of the relaxation to equilibrium of solutions to the Boltzmann equation \cite{CerIllPul,UkaYan,Vil}, whose initial investigations can be traced back to Boltzmann's work \cite{Bol} containing the first formulation of the $H$-theorem for a single-species gas.

\smallskip
To provide a brief context here, from a mesoscopic point of view we may describe the evolution of a dilute gas, composed of a large number of identical monatomic particles undergoing binary collisions, with the use of a distribution function $F=F(t,x,v)$ that satisfies, on $\R_+\times\T^3\times\R^3$, the monospecies Boltzmann equation
\begin{equation}\label{eq:mono BE}
\partial_t F +v\cdot\nabla_x F = Q(F,F),
\end{equation}
where the Boltzmann integral operator $Q$, acting only on the velocity variable, models the microscopic collision processes. Now, from the entropy dissipation $\int_{\R^3} Q(F,F) \log F\dd v \leq 0$ associated with $Q$ follows the decrease over time of the entropy functional $H(F) = \int_{\R^3} F \log F \dd v$, until the distribution $F$ relaxes to a local (in time and space) Maxwellian equilibrium state, solution of $Q(F,F) = 0$. When the influence of the transport operator is considered, the $H$-theorem then implies for large time asymptotics the convergence of solutions to \eqref{eq:mono BE} toward a global equilibrium state having the form of a uniform (in time and space) Maxwellian distribution $M_{\infty}(v)=c_{\infty}(2\pi K_B T_{\infty})^{-3/2}\exp\left( -\frac{|v-u_{\infty}|^2}{2 K_B T_{\infty}} \right)$, with $c_{\infty}, T_{\infty}>0$ and $u_{\infty}\in\R^3$. Here, $K_B$ is the Boltzmann gas constant and the quantities $c_{\infty}$, $u_{\infty}$ and $T_{\infty}$ denote the global concentration, velocity and temperature of the gas, uniquely determined by the initial data.

\smallskip
Understanding whether the equilibrium is reached reasonably fast is a central question in kinetic theory of gases, as one expects that the approximation provided by the chaos molecular assumption is legitimate only for a time of order at most $\mathcal{O}(N_A)$, where $N_A=10^{23}$ is the Avogadro's number. Indeed, by this time each particle will have collided with a nonnegligible fraction of the other atoms and thus the validity of the Boltzmann equation should break down (the reader may refer to the discussion presented in \cite[Chapter 1, Section 2.5]{Vil}). It is therefore crucial to obtain explicit quantitative estimates on the time scale of the convergence, in order to show that it is much smaller than the time scale of validity of the model. 

\smallskip
From the point of view of the linearized theory, the problem reduces to investigating the behavior of a small perturbation $f$ around the global Maxwellian distribution $M_\infty$, and the solution to \eqref{eq:mono BE} is recovered as $F = M_\infty + f$. In fact, by injecting this expansion into the collision term $Q$, it is possible to prove an equivalent version of the $H$-theorem for the linearized Boltzmann operator $L_{M_\infty}(f) = Q(M_\infty, f) + Q(f, M_\infty)$ stating that the associated entropy dissipation, or Dirichlet form, $D(f) = \int_{\R^3} L_{M_\infty}(f) f M_\infty^{-1}\dd v$ is nonpositive and satisfies the upper bound
\begin{equation} \label{eq:mono SG}
D(f) \leq - \lambda \norm{f - \pi(f)}_{L_v^2(M_\infty^{-1/2})}^2,
\end{equation}
for any $f\in L^2(\R^3, M_\infty^{-1/2}) = \Big\{f : \R^3\to\R\ \textrm{measurable s.t. } \norm{ f }_{L_v^2(M_\infty^{-1/2})}^2 = \int_{\R^3} f^2 M_\infty^{-1}\dd v < +\infty\Big\}$. Here, $\pi$ is the orthogonal projection onto the space $\Ker (L_{M_\infty})$ of the equilibrium states of the linearized operator, while $\lambda > 0$ denotes the spectral gap of $- L_{M_\infty}$. In a spatially homogeneous setting, this approach was initiated by Hilbert \cite{Hil}, who studied the properties of $L_{M_\infty}$ in the hard-potential case. Existence of $\lambda$ was then proved by Carleman \cite{Car} and Grad \cite{Gra1,Gra2} for Maxwellian, hard-potential and hard-sphere collision kernels with cutoff, while the whole non-homogeneous case was successfully tackled by Ukai \cite{Uka}, who derived the spectral gap estimate \eqref{eq:mono SG} for the linear operator $v \cdot \nabla_x - L_{M_\infty}$. However, none of these results provides any information on the magnitude of $\lambda$ (nor on its dependence on the initial datum and on the physical quantities appearing in the problem) since they rely on the non-constructive argument that the essential spectrum of the coercive part of $L_{M_\infty}$ is preserved under the action of the remaining compact component, thanks to Weyl's theorem \cite{Kat}. 

\smallskip
The construction of a quantitative theory of convergence to equilibrium for solutions to the Boltzmann equation \eqref{eq:mono BE} that are already close to a steady state, in this linearized setting, translates into the derivation of explicit estimates for the spectral gap $\lambda$ and is known to constitute a first fundamental step for determining the effective relaxation rates in the full nonlinear case \cite{Mou2}, when combined with far-from-equilibrium quantitative entropy dissipation estimates \cite{DesVil}. The first outcomes in this sense were established in the case of Maxwell molecules via an explicit diagonalization of the linearized Boltzmann operator \cite{Bob,WanUhldeB}. In more recent years, systematic derivations of explicit hypocoercivity estimates for the linearized collision operator have been obtained for general cutoff and non-cutoff collision kernels \cite{BarMou,Mou1,MouStr}, finally leading to the first quantitative results on the relaxation to equilibrium in the full nonhomogeneous setting \cite{Mou2,MouNeu}. Notice that these derivations have also opened the way to the recovery of explicit convergence rates toward the hydrodynamic regimes of the Boltzmann equation, applied in combination with the energy method \cite{GuoINS} or the hypocoercive approach \cite{Bri}.

\smallskip
We are interested here in examining this question in the framework of the kinetic theory of multicomponent reacting gases \cite{AnwBisSalSoa,BruSch,Gio1,Gio2,GioGra,GroRosSpi,RosMaz,RosSpi}. Without going into too much details, we consider a mixture of 4 different species $\mathcal{S}_i$, $1 \leq i \leq 4$, whose distribution function $\mathbf{F} = (F_1, \ldots, F_4)$ solves, on $\R_+\times\T^3\times\R^3$, the Boltzmann-like equation
\begin{equation}\label{eq:reactive BE 1}
\partial_t \mathbf{F} + v \cdot \nabla_x \mathbf{F} = \mathbf{Q}(\mathbf{F},\mathbf{F}),
\end{equation}
where the Boltzmann operator $\mathbf{Q} = \mathbf{Q}^\EL + \mathbf{Q}^\CH$ splits into an inert part $\mathbf{Q}^\EL$ modeling microscopic elastic binary collisions between the particles and a reactive term $\mathbf{Q}^\CH$ describing reversible bimolecular chemical reactions of the form $\mathcal{S}_1 + \mathcal{S}_2 \leftrightharpoons \mathcal{S}_3 + \mathcal{S}_4$, with chemical binding energy $E_{12}^{34}$. Similarly to single-species gases, a multicomponent version of the $H$-theorem holds for the reactive operator $\mathbf{Q}$, see e.g. \cite{DesMonSal}, but in this case the corresponding $H$-function splits into two components associated respectively with the operators $\mathbf{Q}^\EL$ and $\mathbf{Q}^\CH$. These separately describe the entropy dissipation coming from both elastic collisions and chemical reactions. As a result, in the large time asymptotics the distribution function $\mathbf{F}$ should converge toward a global equilibrium $\mathbf{M}_\infty = (M_{1,\infty}, \ldots, M_{4,\infty})$ whose Maxwellian shape
\begin{equation} \label{eq:Maxwellian}
M_{i, \infty}(v) = c_{i,\infty} \left( \frac{m_i}{2 \pi K_B T_\infty} \right)^{3/2} \exp \left( -m_i \frac{|v - u_{\infty}|^2}{2 K_B T_\infty} \right), \quad v \in \R^3, \ 1 \leq i \leq 4,
\end{equation}
is prescribed by the elastic part $\mathbf{Q}^\EL$, while the chemical operator $\mathbf{Q}^\CH$ imposes the additional mass action law constraint
\begin{equation} \label{eq:MAL 1}
c_{3,\infty} c_{4,\infty} = c_{1,\infty} c_{2,\infty} \left( \frac{m_1 m_2}{m_3 m_4} \right)^{-3/2} \exp \left( -\frac{E_{12}^{34}}{K_B T_{\infty}} \right).
\end{equation}
Here, the vectors $(m_i)_{1\leq i\leq 4}$ and $(c_{i,\infty})_{1\leq i\leq 4} \in (\R_+^*)^4$ stand for the masses and the global concentrations of the different species, while $u_\infty \in \R^3$ and $T_\infty > 0$ indicate the bulk velocity and temperature of the whole mixture.

\smallskip
Our aim is to provide quantitative information on the spectral gap of the multi-species operator linearized around global Maxwellians like \eqref{eq:Maxwellian}. We will do so by considering a regime of small fluctuations $\mathbf{f} = (f_1, \ldots, f_4)$ of the equilibrium $\mathbf{M}_\infty$ and by deriving an explicit upper bound similar to \eqref{eq:mono SG}, for the Dirichlet form associated with the operator $\mathbf{L}_{\mathbf{M}_\infty}(\mathbf{f}) = \mathbf{Q}(\mathbf{M}_\infty,\mathbf{f}) + \mathbf{Q}(\mathbf{f}, \mathbf{M}_\infty)$. This contribution is still missing, despite being at the core of follow-up investigations on the speed of relaxation toward the equilibrium of reactive mixtures \cite{DesMonSal}, as well as on the rate of convergence to hydrodynamic models obtained from properly rescaled versions of \eqref{eq:reactive BE 1}, including reaction--diffusion systems \cite{BisDes} and reactive Maxwell--Stefan equations \cite{AnwBisSalSoa,AnwGonSoa}. In the multi-species kinetic literature there exist in fact several recent papers revolving around this topic, but they are mostly exploiting a non-constructive approach based on Weyl's perturbation theorem. We mention for example the first compactness result for the linearized collision operator modeling inert monatomic mixtures \cite{BouGrePavSal} and similar strategies developed to analyse more elaborate linearized operators for collisions inside of polyatomic gases \cite{Ber1,Ber2,Ber3, Ber4,BorBouSal,BruShaThi}, also involving chemical reactions \cite{Ber5,CarPolSoa}. The interested reader can consult the reviews \cite{BerBouColGre,BouSal} detailing the literature and techniques revolving around these recent outcomes. Results on quantitative coercivity estimates are instead less prevalent, pertain solely to the non-reactive setting (and to the case of Maxwellian, hard-potential and hard-sphere collision kernels with cutoff) and can be narrowed down to the works \cite{BriDau,DauJunMouZam}, which first obtained explicit bounds on the spectral gap of the linearization of $\mathbf{Q}^\EL$, and \cite{BonBouBriGre} where it was proved that this spectral gap is stable under small local non-equilibrium perturbations of the global Maxwellian $\mathbf{M}_\infty$, in connection with the rigorous derivation of the inert Maxwell--Stefan model \cite{BonBri1,BonBri2}.

\smallskip
Following the ideas developed in \cite{BriDau,DauJunMouZam}, we approach the problem by tackling separately the linearized counterparts of $\mathbf{Q}^\EL$ and $\mathbf{Q}^\CH$. Starting from the explicit coercivity estimate that is known to hold for the linearized multi-species elastic operator, we show that its spectral gap is not perturbed too much by the presence of the reactive component $\mathbf{Q}^\CH$. A careful treatment of the intricate cross-effects introduced by the chemical reactions is then needed in order to understand how they modify the distribution functions belonging to the set of equilibria of the inert part. Specifically, we construct a nonpositive functional over this nullspace, which cancels out at the local equilibria prescribed by (a linearized version of) the constraint \eqref{eq:MAL 1} and provides a negative upper bound for the Dirichlet form associated with the linearized chemical operator, allowing to control the cross-effects.

\medskip
The paper is organized as follows. In Section 2 we begin by introducing the kinetic model and the assumptions that are made on the different collision kernels. In particular, we provide an explicit characterization of a large class of cutoff hard-potential chemical cross-sections satisfying the micro-reversibility property \cite{GroPol}. We proceed by presenting the linearization of the collision operators and continue, in Section 3, recalling known properties of the inert and reactive parts. Then, we will state our main theorem and provide a brief description of the strategy used to obtain it. At last, the final section is devoted to the proof of this result.

%%%%%%%%%%%%%%%%%%%%%%%%%%%%%%%%%%%%%%%%%%%%%%%%%%%%%%%%%%%%
%%%%%%%%%%%%%%%%%%%%%%%%%%%%%%%%%%%%%%%%%%%%%%%%%%%%%%%%%%%%
%%%%%%%%%%%    SECTION 2: THE KINETIC MODEL    %%%%%%%%%%%%%
%%%%%%%%%%%%%%%%%%%%%%%%%%%%%%%%%%%%%%%%%%%%%%%%%%%%%%%%%%%%
%%%%%%%%%%%%%%%%%%%%%%%%%%%%%%%%%%%%%%%%%%%%%%%%%%%%%%%%%%%%
	
\section{The reactive kinetic model}

\noindent We consider a dilute $4$-species gaseous mixture of particles interacting at the microscopic level via elastic binary collisions and reversible chemical reactions. The evolution of the different species, characterized by their respective molecular masses $(m_i)_{1 \leq i \leq 4}$, is described by a vector distribution function $\mathbf{F} = (F_1, \ldots, F_4)$ whose components solve, over time $t > 0$, space $x \in \T^3$ and velocity $v \in \R^3$, the system of reactive Boltzmann equations
\begin{equation}\label{eq:reactive BE 2}
\partial_t F_i + v \cdot \nabla_x F_i = Q_i(\mathbf{F},\mathbf{F}), \quad 1 \leq i \leq 4.
\end{equation}
Notice that vectors and vector-valued functions in the species will always be denoted by bold letters, while the corresponding indexed letters will indicate their components. For example, $\mathbf{W}$ stands for the vector or vector-valued function $(W_1, \ldots, W_4)$.

\smallskip
The Boltzmann multi-species operator $\mathbf{Q}(\mathbf{F}, \mathbf{F}) = \big( Q_1(\mathbf{F}, \mathbf{F}), \ldots, Q_4(\mathbf{F}, \mathbf{F}) \big)$ has a quadratic integral form and models the elastic and reactive interactions between the species. Therefore, it can be split into two separate parts as $\mathbf{Q}(\mathbf{F}, \mathbf{F}) = \mathbf{Q}^\EL(\mathbf{F}, \mathbf{F}) + \mathbf{Q}^\CH(\mathbf{F}, \mathbf{F})$, which only act on the velocity variable $v$ and are thus local in $(t, x)$.

%%%%%%%%%%%%%%%%%%%    ELASTIC OPERATOR    %%%%%%%%%%%%%%%%%%%%
\medskip
\noindent \textbf{The elastic component.} The operator $\mathbf{Q}^\EL$ gives a balance of the binary elastic collisions between particles of the same or of different species and is defined component-wise, for any $1 \leq i \leq 4$, by
\begin{equation*}
Q_i^\EL(\mathbf{F},\mathbf{F})(v) = \sum_{j=1}^4 Q_{ij}^\EL(F_i, F_j)(v) = \sum_{j=1}^4 \int_{\R^3\times\Sf} B_{ij}(|v - v_*|, \cos \theta)( F_i^{\prime} F_j^{\prime *} - F_i F_j^* )\dd v_* \dd \sigma, \quad v \in \R^3,
\end{equation*}
where we utilize the standard shorthand notations $F_i^{\prime} = F_i(v^{\prime})$, $F_i = F_i(v)$, $F_j^{\prime *} = F_j(v_*^{\prime})$ and $F_j^* = F_j(v_*)$. The post-collisional velocities $v^{\prime}$ and $v_*^{\prime}$ are given in terms of the pre-collisional velocities $v$ and $v_*$ by the elastic collision rules
\begin{equation*}
v^{\prime} = \frac{m_i v + m_j v_*}{m_i + m_j} + \frac{m_j}{m_i + m_j} |v - v_*| \sigma, \qquad
v_*^{\prime} = \frac{m_i v + m_j v_*}{m_i + m_j} - \frac{m_i}{m_i + m_j} |v - v_*| \sigma,
\end{equation*}
where $\sigma \in \Sf$ is a parameter whose existence is ensured by the conservation of microscopic momentum and kinetic energy
\begin{equation} \label{eq:conservation EL}
m_i v + m_j v_* = m_i v^{\prime} + m_j v_*^{\prime}, \qquad
\frac{1}{2} m_i |v|^2 + \frac{1}{2} m_j |v_*|^2 = \frac{1}{2} m_i |v^{\prime}|^2 + \frac{1}{2} m_j |v_*^{\prime}|^2.
\end{equation}
The elastic collisional cross-sections $B_{ij}$ are nonnegative functions of the modulus of the incoming relative velocity of the colliding particles $|v - v_*|$ and of the cosine of the deviation angle $\theta \in [0,\pi]$ between $v - v_*$ and $\sigma \in \Sf$. They encode the information on how the mixture molecules interact microscopically and their choice is essential for studying the properties of the Boltzmann operator. We focus our attention here on cutoff Maxwellian, hard-potential and hard-sphere collision kernels.

%%%%%%%%%%%%%%%%%%    ELASTIC CROSS-SECTIONS    %%%%%%%%%%%%%%%%%%%
\medskip
\noindent \textbf{Assumptions on the elastic cross-sections.}  The following assumptions on the elastic collision kernels $(B_{ij})_{1\leq i,j \leq 4}$ are the same considered in \cite{BriDau}.

\smallskip
\begin{enumerate}[leftmargin=1.5cm]
\item[(EL1)] They satisfy a symmetry property when interchanging the species indices $i$ and $j$
\begin{equation*}
B_{ij}(|v - v_*|, \cos \theta) = B_{ji}(|v - v_*|, \cos \theta), \quad \forall v, v_* \in \R^3, \ \forall \theta \in [0,\pi].
\end{equation*}
\item[(EL2)] They write as the product of a kinetic part $\Phi_{ij} \geq 0$ and an angular part $b_{ij} > 0$, namely
\begin{equation*}
B_{ij}(|v - v_*|, \cos \theta) = \Phi_{ij}(|v - v_*|) b_{ij}(\cos \theta), \quad \forall v, v_* \in \R^3, \ \forall \theta \in [0,\pi].
\end{equation*}
\item[(EL3)] For the kinetic part, we assume that there exists a constant $\gamma \in [0,1]$ such that
\begin{equation*}
\Phi_{ij}(|v - v_*|) = C_{ij} |v - v_*|^{\gamma}, \quad C_{ij} > 0, \quad \forall v, v_* \in \R^3.
\end{equation*}
\item[(EL4)] We suppose that the angular part is positive for a.e. $\theta \in [0,\pi]$ and that there exists a constant $C >0$ such that, for any $1 \leq i,j \leq 4$,
\begin{equation*}
b_{ij}(\cos \theta) \leq C |\cos \theta| |\sin \theta|, \quad \forall \theta \in [0,\pi].
\end{equation*}
Furthermore, we assume that for any $1\leq i\leq 4$
\begin{equation*}
\inf_{\sigma_1, \sigma_2 \in \Sf} \int_{\Sf} \min\{ b_{ii}(\sigma_1 \cdot \sigma_3), b_{ii}(\sigma_2 \cdot \sigma_3) \} \dd \sigma_3 > 0.
\end{equation*}
\end{enumerate}
Assumption (EL1) translates a micro-reversibility property for collisions. Assumption (EL2) on the tensorised form of $B_{ij}$ is classical in both monospecies and multi-species settings (see \cite{BarMou,BriDau,DauJunMouZam,Mou1}) and holds for all collision kernels describing interaction potentials that behave like inverse-power laws. This hypothesis is used for a sake of simplicity and it could be weakened to consider upper and lower bounds on $B_{ij}$ by tensorial products of this kind. Dismissing this assumption completely would need to be compensated by the introduction of technical conditions on $B_{ij}$, to ensure the validity of the quantitative estimates on the spectral gap of the monospecies linearized Boltzmann operator derived in \cite{BarMou,Mou1}. Assumption (EL3) on the kinetic component is satisfied by Maxwell molecules ($\gamma = 0$), hard-potential ($0 < \gamma < 1$) and hard-sphere ($\gamma = 1$) collision kernels. The two conditions on $b_{ij}$ allow to ensure its integrability on the sphere and the positivity of such integral. In particular, the upper bound on $b_{ij}$ is classical \cite{BouGrePavSal,BriDau,DauJunMouZam,Gra2,Mou1} and implies Grad's cutoff assumption. Finally, the last assumption on $b_{ii}$ in (EL4) is satisfied by all relevant physical models and is needed to recover a spectral gap for the linearized operator associated with $\mathbf{Q}^\EL$ \cite{BriDau,Mou1}.

\begin{remark} \label{remark1}
We point out that one could replace the rather restrictive assumption (EL3) on the kinetic part $\Phi_{ij}(|v-v_*|)$ with the following more general condition:
\begin{enumerate}[leftmargin=1.5cm]
\item[(EL3')] There exist three constants $\underline{C}_{ij}, \overline{C}_{ij} > 0$ and $\gamma_{ij}=\gamma_{ji} \in [0,1]$ such that
\begin{equation*}
\underline{C}_{ij} |v-v_*|^{\gamma_{ij}} \leq \Phi_{ij}(|v - v_*|) \leq \overline{C}_{ij} \big(1 + |v - v_*|\big), \quad \forall v, v_* \in \R^3.
\end{equation*}
\end{enumerate}
However, in this case one must also enforce (see \cite[Section 1.3, condition (A6)]{DauJunMouZam}) that for any $1 \leq i \leq 4$ the elastic multi-species collision frequency
\begin{equation} \label{eq:nu EL}
    \nu_i^\EL(v) = \sum_{j=1}^4 \int_{\R^3 \times \Sf} B_{ij}(|v-v_*|,\cos\theta) \mu_j(v_*) \dd v_* \dd \sigma, \quad v \in \R^3.
\end{equation}
is controlled by the corresponding monospecies component
\begin{equation*}
    \nu_{ii}^\EL(v) = \int_{\R^3 \times \Sf} B_{ii}(|v-v_*|,\cos\theta) \mu_i(v_*) \dd v_* \dd \sigma, \quad v \in \R^3,
\end{equation*}
This amounts to impose the additional condition 
\begin{enumerate}[leftmargin=1.5cm]
\item[(EL5)] There exists a constant $\beta^\EL > 0$ such that, for any $1\leq i \leq 4$, it holds
\begin{equation*}
    \nu_i^\EL(v) \leq \beta^\EL \nu_{ii}^\EL(v), \quad \forall v \in \R^3.
\end{equation*}
\end{enumerate}
Hypothesis (EL3') would allow to consider more general collision kernels of power law-type that possess different species-dependent exponents \cite{AloColGam,BouGrePavSal} of the form $\Phi_{ij}(|v-v_*|) = C_{ij} |v-v_*|^{\gamma_{ij}}$, with $\gamma_{ij} \in [0,1]$. We emphasize that the lower bound in (EL3') is a minimal requirement needed to recover quantitative estimates on the spectral gap of the inert Boltzmann operator \cite{DauJunMouZam}. The same consideration holds for the upper bound, which imposes that the kinetic part is of restricted growth for both small and large values of $|v-v_*|$. We stress however that (EL3') alone is not enough to obtain explicit controls on the spectral gap of the multi-species Boltzmann operator \cite{DauJunMouZam} and must be combined with assumption (EL5) to guarantee the validity of Theorem \ref{theorem:elastic SG}. This hypothesis ensures that the ratio between off-diagonal and diagonal collision frequencies can be bounded uniformly from above, and that one can use the spectral gap of the (inert) monospecies operator to control that of the multi-species operator. In \cite{DauJunMouZam} this condition is actually imposed by requiring the existence of a constant $\beta^\EL > 0$ such that $B_{ij}(|v-v_*|,\cos\theta) \leq \beta^\EL B_{ii}(|v-v_*|,\cos\theta)$ for any $v,v_* \in \R^3$ and $\theta \in [0,\pi]$, but the authors there deal with a particular multi-species setting where all molecular masses $m_i = m$, $1 \leq i \leq 4$, are equal. Here, we chose to consider (EL3), instead of (EL3')--(EL5), for a sake of clarity in the derivation of the control estimates on the collision frequencies of the Boltzmann reactive operator, but our results could be adapted from \cite{DauJunMouZam} to deal with the more general setting provided by (EL3')--(EL5). Notice in particular that (EL3) implies not only (EL3') but also (EL5), since in this case it would simply hold $\nu_i^\EL(v) \sim \nu_{ii}^\EL(v) \sim 1+|v|^\gamma$ for any $1 \leq i \leq 4$.
\end{remark}

%%%%%%%%%%%%%%%%%%%    CHEMICAL OPERATOR    %%%%%%%%%%%%%%%%%%%
\medskip
\noindent \textbf{The reactive component.} To model the reactive part, we follow the framework introduced by Rossani and Spiga in \cite{RosSpi}. The operator $\mathbf{Q}^\CH = (Q_1^\CH, \ldots, Q_4^\CH)$ describes the way particles of different species interact through the reversible bimolecular chemical reaction
\begin{equation} \label{eq:reaction}
\mathcal{S}_1 + \mathcal{S}_2 \leftrightharpoons \mathcal{S}_3 + \mathcal{S}_4,
\end{equation}
where the total mass of the components involved, $m_1 + m_2 = m_3 + m_4$, is conserved in the process. 
%\revise{A typical reversible reaction of the form \eqref{eq:reaction} is the esterification}
%\begin{equation*}
%	\revise{\text{CH}_3\text{CO}_2\text{H} + \text{C}_2\text{H}_5\text{OH} \leftrightharpoons \text{CH}_3\text{CO}_2\text{C}_2\text{H}_5 + \text{H}_2\text{O}.}
%\end{equation*}
We assume here that the forward reaction is endothermic with impinging energy $E_{12}^{34} = E_4 + E_3 - E_2 - E_1 \geq 0$, where $E_i$ denotes the energy of the chemical link of the species $\mathcal{S}_i$, $1 \leq i \leq 4$. It will be useful to also denote with $E_{34}^{12} = E_2 + E_1 - E_4 - E_3 \leq 0$ the energy dissipated in the backward exothermic process. In general, we shall always use the notation $W_{ij}^{hk}$, where the quadruples $(i,j,h,k)$ cover all the possible combinations of indices for the reaction \eqref{eq:reaction}, to indicate that the variable $W$ is related to the specific interaction $\mathcal{S}_i + \mathcal{S}_j \to \mathcal{S}_h + \mathcal{S}_k$.

\smallskip
We define each $Q_i^\CH$, $1 \leq i \leq 4$, separately. For the forward reaction $\mathcal{S}_1 + \mathcal{S}_2 \to \mathcal{S}_3 + \mathcal{S}_4$, the post-collisional velocities $v^\prime = v_{12}^{34}$ and $v^\prime_* = v_{*12}^{\ 34}$ depend on the velocities $v$ and $v_*$ of the incoming particles (belonging to the species $\mathcal{S}_1$ and $\mathcal{S}_2$) through the relations
\begin{equation} \label{eq:post-collisional CH}
v^\prime = \frac{m_1 v + m_2 v_*}{m_1 + m_2} + \frac{m_4}{m_3 + m_4} g_{12}^{34} \sigma, \qquad
v^\prime_* = \frac{m_1 v + m_2 v_*}{m_1 + m_2} - \frac{m_3}{m_3 + m_4} g_{12}^{34} \sigma,
\end{equation}
where we have denoted with $g_{12}^{34} = |v^\prime - v^\prime_*|$ the modulus of the outgoing relative velocity
\begin{equation} \label{eq:energy}
g_{12}^{34} = \left[ \frac{m_{12}}{m_{34}} \left( |v - v_*|^2 - \frac{2 E_{12}^{34}}{m_{12}} \right) \right]^{1/2},
\end{equation}
and with $m_{ij} = \frac{m_i m_j}{m_i + m_j}$ the various reduced masses. In particular, the relation \eqref{eq:energy} can be deduced by the conservation of momentum and total (kinetic and chemical) energy
\begin{equation} \label{eq:conservation CH}
\begin{aligned}
m_1 v + m_2 v_* & = m_3 v^\prime + m_4 v^\prime_*, \\[2mm]
\frac{1}{2} m_1 |v|^2 + E_1 + \frac{1}{2} m_2 |v_*|^2 + E_2 & = \frac{1}{2} m_3 |v^\prime|^2 + E_3 + \frac{1}{2} m_4 |v^\prime_*|^2 + E_4,
\end{aligned}
\end{equation}
which also imply the existence of the parameter $\sigma \in \Sf$. It is then clear that for the endothermic reaction to happen, there must be enough impinging energy ensuring that $|v - v_*|^2 \geq 2 E_{12}^{34} /m_{12}$. Bearing this last observation in mind, the net production of molecules of species $\mathcal{S}_1$ due to the chemical reaction \eqref{eq:reaction} is prescribed by the operator $Q_1^\CH$ which writes
\begin{equation*} %\label{Q Chemical 1}
\begin{aligned}
Q_1^\CH(\mathbf{F}, \mathbf{F})(v) & = \int_{\R^3 \times \Sf} H \left( |v - v_*|^2 - \frac{2 E_{12}^{34}}{m_{12}} \right) B_{12}^{34}(|v - v_*|, \cos \theta)
\\[2mm]   & \hspace{2.5cm} \times \left[ \left( \frac{m_{12}}{m_{34}} \right)^3 F_3(v^\prime) F_4 (v^\prime_*) - F_1(v) F_2(v_*) \right] \dd v_* \dd \sigma,
\end{aligned}
\end{equation*}
with the threshold on the impinging energy encoded by the Heaviside function $H(x) = \mathds{1}_{x \geq 0}$. Inside the integral appears the cross-section $B_{12}^{34}$ relative to the chemical process $\mathcal{S}_1 + \mathcal{S}_2 \to \mathcal{S}_3 + \mathcal{S}_4$. Like the elastic cross-sections, all the reactive kernels $B_{ij}^{hk}$ are nonnegative functions that only depend on the quantities $|v - v_*|$ and $\cos \theta = \frac{v - v_*}{|v - v_*|} \cdot \sigma$. By assuming the natural indistinguishability condition for the $B_{ij}^{hk}$, namely
\begin{equation} \label{eq:indistinguishability}
B_{ij}^{hk}(|v - v_*|, \cos \theta) = B_{ji}^{kh}(|v - v_*|, \cos \theta),
\end{equation}
the net production of molecules of species $\mathcal{S}_2$ is given by the operator $Q_2^\CH$, obtained from $Q_1^\CH$ through a permutation of indices as
\begin{equation*} %\label{Q Chemical 2}
\begin{aligned}
Q_2^\CH(\mathbf{F}, \mathbf{F})(v) & = \int_{\R^3 \times \Sf} H \left( |v - v_*|^2 - \frac{2 E_{12}^{34}}{m_{12}} \right) B_{12}^{34}(|v - v_*|, \cos \theta)
\\[2mm]   & \hspace{2.5cm} \times \left[ \left( \frac{m_{12}}{m_{34}} \right)^3 F_4(v^\prime) F_3 (v^\prime_*) - F_2(v) F_1(v_*) \right] \dd v_* \dd \sigma,
\end{aligned}
\end{equation*}
where now the velocities $v^\prime = v_{21}^{43}$ and $v^\prime_* = v_{* 21}^{\ 43}$ are obtained by simply exchanging the indices $1 \leftrightarrow 2$ and $3\leftrightarrow 4$ in the expressions \eqref{eq:post-collisional CH}.

\smallskip
Since the process $\mathcal{S}_3 + \mathcal{S}_4 \to \mathcal{S}_1 + \mathcal{S}_2$ is exothermic, there is no appearance of the previous energy threshold \eqref{eq:energy} and thus the backward reaction can occur for any relative speed of the incoming velocities. If we now assume that the forward and backward reactive cross-sections are related by the so-called micro-reversibility condition
\begin{equation} \label{eq:micro-reversibility}
|v - v_*| B_{ij}^{hk}(|v - v_*|, \cos \theta) = \left( \frac{m_{hk}}{m_{ij}} \right)^2 |v_{ij}^{hk} - v_{* ij}^{\ hk}| B_{hk}^{ij}(|v_{ij}^{hk} - v_{* ij}^{\ hk}|, \cos \theta),
\end{equation}
then, the operator $Q_3^\CH$ modeling the net production of molecules of species $\mathcal{S}_3$ has the form
\begin{equation*}
Q_3^\CH(\mathbf{F}, \mathbf{F})(v) = \int_{\R^3 \times \Sf} B_{34}^{12}(|v - v_*|, \cos \theta) \left[ \left( \frac{m_{34}}{m_{12}} \right)^3 F_1(v^\prime) F_2(v^\prime_*) - F_3(v) F_4(v_*) \right] \dd v_* \dd \sigma,
\end{equation*}
with $v^\prime = v_{34}^{12}$ and $v^\prime_* = v_{* 34}^{\ 12}$. At last, the balance of molecules of species $\mathcal{S}_4$ produced in the backward reaction is given by the operator $Q_4^\CH$, which is obtained from $Q_3^\CH$ by permuting the indices and using condition \eqref{eq:indistinguishability}, and reads
\begin{equation*}
Q_4^\CH(\mathbf{F}, \mathbf{F})(v) = \int_{\R^3 \times \Sf} B_{34}^{12}(|v - v_*|, \cos \theta) \left[ \left( \frac{m_{34}}{m_{12}} \right)^3 F_2(v^\prime) F_1(v^\prime_*) - F_4(v) F_3(v_*) \right] \dd v_* \dd \sigma,
\end{equation*}
by an obvious permutation of the indices in the corresponding post-collisional velocities to determine the expressions of $v^\prime = v_{43}^{21}$ and $v^\prime_* = v_{* 43}^{\ 21}$.

\smallskip
At this point we need to state the precise assumptions on the reactive kernels $B_{ij}^{hk}$, that are required for the following analysis. Motivated by the result obtained in \cite{GroPol} for hard-sphere interactions, we provide a first (up to our knowledge) explicit characterization of a wide class of cutoff hard-potential reactive kernels compatible with the micro-reversibility condition \eqref{eq:micro-reversibility}. Similar considerations appear in fact in many papers dealing with the study of polyatomic gases \cite{Ber1, Ber2,Ber3,Ber4,Ber5,BorBouSal,BruShaThi,GamPav}, but it seems that a general framework able to include chemical reactions is still missing and we believe that this contribution, although very small, could be of independent interest for the researchers in the field.

%%%%%%%%%%%%%%%%%    CHEMICAL CROSS-SECTIONS    %%%%%%%%%%%%%%%%%%%
\medskip
\noindent \textbf{Assumptions on the reactive cross-sections.} For a sake of clarity in the presentation, let us introduce the following unified notation for the incoming and the outgoing relative velocities
\begin{equation*}
g = |v - v_*|, \qquad g_{ij}^{hk} = \left[ \frac{m_{ij}}{m_{hk}} \left( g^2 - \frac{2 E_{ij}^{hk}}{m_{ij}} \right) \right]^{\frac{1}{2}},
\end{equation*}
corresponding to any of the interactions in \eqref{eq:reaction}. Notice in particular that for the reverse reactions no threshold energy is required since $- E_{34}^{12} = E_{12}^{34} \geq 0$. When choosing the cross-sections in the chemical framework (see \cite{GroPol} for an exhaustive discussion on this matter), we need to ensure that they satisfy condition \eqref{eq:micro-reversibility} which rewrites, using this notation, as
\begin{equation*}
g B_{ij}^{hk}(g, \cos \theta) = \left( \frac{m_{hk}}{m_{ij}} \right)^2 g_{ij}^{hk} B_{hk}^{ij}(g_{ij}^{hk}, \cos \theta).
\end{equation*}
In order to identify a reasonable structure for the $B_{ij}^{hk}$ to solve these relations, the key observation is that the reversibility of the reaction \eqref{eq:reaction} translates into a symmetry property on the outgoing relative velocities
\begin{equation*}
g_{hk}^{ij} \big |_{g_{ij}^{hk}} = \left[ \frac{m_{hk}}{m_{ij}} \left( \big( g_{ij}^{hk} \big)^2 + \frac{2 E_{ij}^{hk}}{m_{hk}} \right)\right]^{\frac{1}{2}} = g,
\end{equation*}
that is easily checked by computations.

\smallskip
Whenever all the reactions $\mathcal{S}_1 + \mathcal{S}_2 \leftrightharpoons \mathcal{S}_3 + \mathcal{S}_4$ are well-defined, i.e. when $g^2 \geq 2 E_{12}^{34} / m_{12}$, the structure of the chemical cross-sections $B_{ij}^{hk}$ is prescribed by the following assumptions.
%that hold for any $v, v_* \in \R^3$ and any $\theta \in [0,\pi]$.

\smallskip
\begin{enumerate}[leftmargin=1.5cm]
\item[(CH1)] They satisfy an indistinguishability condition in the interchange of the indices $i \leftrightarrow j$ and $h \leftrightarrow k$, namely
\begin{equation*}
B_{ij}^{hk}(g, \cos \theta) = B_{ji}^{kh}(g, \cos \theta), \quad \forall v, v_* \in \R^3, \ \forall \theta \in [0,\pi].
\end{equation*}
\item[(CH2)] They decompose into the product of a kinetic part $\Phi_{ij}^{hk}\geq 0$ and an angular part $b_{ij}^{hk} \geq 0$, namely
\begin{equation*}
B_{ij}^{hk}(g, \cos \theta) = \Phi_{ij}^{hk}(g)b_{ij}^{hk}(\cos \theta), \quad \forall v, v_* \in \R^3, \ \forall \theta \in [0,\pi].
\end{equation*}
\item[(CH3)] The kinetic part has the form of reactive Maxwellian or hard-potential interactions, i.e. for any $\gamma \in [0,1]$,
\begin{equation*}
\Phi_{ij}^{hk}(g) = C_{ij}^{hk} \frac{\left( g_{ij}^{hk} \right)^{\frac{\gamma + 1}{2}}}{g^{\frac{1-\gamma}{2}}},\quad C_{ij}^{hk} > 0, \quad C_{ij}^{hk} m_{ij}^2 = C_{hk}^{ij} m_{hk}^2, \quad \forall v, v_* \in \R^3.
\end{equation*}
\item[(CH4)] The angular part is assumed to be positive for a.e. $\theta \in [0,\pi]$ and to satisfy the condition $b_{ij}^{hk}(\cos \theta) = b_{hk}^{ij}(\cos \theta)$. Moreover, there exists a constant $C >0$ such that
\begin{equation*}
b_{ij}^{hk}(\cos \theta) \leq C |\cos \theta| |\sin \theta| \quad \forall \theta \in [0,\pi].
\end{equation*}
\end{enumerate}
The first hypothesis (CH1) translates part of the micro-reversibility property for reactions and is very natural. The second one (CH2) provides a simple way to exhibit an explicit structure for the chemical kernels. Assumption (CH3) translates, in the reactive framework, the proper shape of Maxwell and inverse-power interactions. Moreover, the condition relating the constants $C_{ij}^{hk}$ and the reduced masses allows to recover the micro-reversibility \eqref{eq:micro-reversibility}, in combination with the symmetries imposed on the angular part in (CH4). Finally, condition (CH4) translates a cutoff property ensuring, like in the elastic case, that $b_{ij}^{hk}$ is integrable over $\Sf$.

\begin{remark} \label{remark2}
    Unlike the elastic setting where the indistinguishability condition (EL1) alone guarantees that the collision kernels $B_{ij}$ have enough symmetries to link the operators $Q_{ij}^\EL$ and $Q_{ji}^\EL$, the corresponding condition (CH1) on the chemical collision kernels $B_{ij}^{hk}$ is not sufficient to pass from $Q_1^\CH$ to $Q_3^\CH$ and one needs to impose the additional micro-reversibility relation \eqref{eq:micro-reversibility}. Here, our goal is to characterize an explicit form of hard-potential-type kernels $B_{ij}^{hk}$ that naturally satisfy condition \eqref{eq:micro-reversibility} and we have done so by exploiting the symmetry properties of the pre and post-collisional relative velocities to deduce (CH3), which prescribes an explicit structure for the kinetic part. If one wanted to be less restrictive on the chemical kernels and consider more general hypotheses like (EL3') and (EL5) in the elastic framework (see our discussion in Remark \ref{remark1}), it would then be necessary to enforce \eqref{eq:micro-reversibility}. Specifically, we could replace assumption (CH3) on the kinetic part $\Phi_{ij}^{hk}(g)$ with the following:
    \begin{enumerate}[leftmargin=1.5cm]
    \item[(CH3')] There exist three constants $\underline{C}_{ij}^{hk}, \overline{C}_{ij}^{hk} > 0$ and $\gamma_{ij}^{hk} \in [0,1]$ such that 
    \begin{equation*}
    \underline{C}_{ij}^{hk} g^{\gamma_{ij}^{hk}} \leq \Phi_{ij}^{hk}(g) \leq \overline{C}_{ij}^{hk} \big(1 + g\big), \quad \forall v, v_* \in \R^3.
    \end{equation*}
    \end{enumerate}
    Then, the chemical collision frequencies $(\nu_i^\CH)_{1\leq i \leq 4}$ given by \eqref{eq:nu CH} should be controlled by their elastic counterparts $(\nu_i^\EL)_{1\leq i \leq 4}$ given by \eqref{eq:nu EL}, through the additional condition:
    \begin{enumerate}[leftmargin=1.5cm]
    \item[(CH5)] There exists a constant $\beta^\CH > 0$ such that, for any $1\leq i \leq 4$, it holds
    \begin{equation*}
        \nu_i^\CH(v) \leq \beta^\CH \nu_i^\EL(v), \quad \forall v \in \R^3.
    \end{equation*}
    \end{enumerate}
    At last, as pointed out above, we must make sure that micro-reversibility holds. Because of (CH4), this is verified as soon as:
    \begin{enumerate}[leftmargin=1.5cm]
    \item[(CH6)] The functions $\Phi_{ij}^{hk}$ satisfy condition \eqref{eq:micro-reversibility}, namely
    \begin{equation*}
        g \Phi_{ij}^{hk}(g) = \left( \frac{m_{hk}}{m_{ij}} \right)^2 g_{ij}^{hk} \Phi_{hk}^{ij}(g_{ij}^{hk}), \quad \forall v, v_* \in \R^3.
    \end{equation*}
    \end{enumerate}
    Like with the elastic case, the additional hypothesis (CH5) essentially translates a property of equivalence between chemical and elastic collision frequencies, ensuring that the spectral gap of $\mathbf{L}^\EL$ can be used to control the Dirichlet form of $\mathbf{L}^\CH$. Here, we chose to follow a simpler approach by considering the slightly more restrictive condition (CH3), in place of (CH3')--(CH5)--(CH6), to guarantee that $\nu_i^\CH(v) \sim \nu_i^\EL(v) \sim 1+|v|^\gamma$ for any $1 \leq i \leq 4$ and make the estimates more transparent.
\end{remark}

%%%%%%%%%%%%%%%%%%    CONSERVATION LAWS    %%%%%%%%%%%%%%%%%%%
\medskip
\noindent \textbf{Collisional invariants and global equilibria.} The following are well-known conservation properties satisfied by the Boltzmann operators $\mathbf{Q}^\EL$ and $\mathbf{Q}^\CH$ \cite{BriDau,CerIllPul,DauJunMouZam,DesMonSal,RosSpi}. Let us consider any vector-valued function $\pmb{\psi} = (\psi_1, \ldots, \psi_4) : \R^3 \to \R^4$ for which the integrals in the equalities below are well-defined. Thanks to intrinsic symmetries of the elastic part, by denoting $\kappa = \frac{v-v_*}{|v-v_*|}$, one can apply standard changes of variable $(v,v_*,\sigma) \mapsto (v^{\prime},v_*^{\prime},\kappa)$ and $(v,v_*,\sigma) \mapsto (v_*,v,-\sigma)$ to prove that the operators $Q_i^\EL(\mathbf{F}, \mathbf{F})$ satisfy the weak formulation
\begin{multline} \label{eq:weak formulation EL}
\sum_{i=1}^4 \int_{\R^3} Q_i^\EL(\mathbf{F}, \mathbf{F})(v) \psi_i(v) \dd v = - \frac{1}{4} \sum_{i=1}^4 \sum_{j=1}^4\int_{\R^6 \times \Sf} B_{ij}(|v - v_*|, \cos \theta) \big( F_i^\prime F_j^{\prime *} - F_i F_j^* \big)
\\[4mm]    \times \big( \psi_i(v^\prime) + \psi_j(v_*^\prime) - \psi_i(v) - \psi_j(v_*) \big) \dd v_* \dd v \dd \sigma.
\end{multline}
The micro-reversibilities satisfied by the chemical component used in tandem with a combination of changes of variables of the form $(v,v_*,\sigma) \mapsto (v_{ij}^{hk},v_{* ij}^{\ hk},\kappa)$ and $(v,v_*,\sigma) \mapsto (v_*,v,-\sigma)$, depending on the combinations of indices $(i, j, h, k)$, allows to argue in a similar way \cite{GroPol} that the operators $Q_i^\CH(\mathbf{F}, \mathbf{F})$ satisfy the weak formulation
\begin{multline} \label{eq:weak formulation CH}
\sum_{i=1}^4 \int_{\R^3} Q_i^\CH(\mathbf{F}, \mathbf{F})(v) \psi_i(v) \dd v = - \int_{\R^6 \times \Sf} H \left( |v - v_*|^2 - \frac{2 E_{12}^{34}}{m_{12}} \right) B_{12}^{34}(|v - v_*|, \cos \theta)
\\[4mm]    \times \left[ \left( \frac{m_{12}}{m_{34}} \right)^3 F_3(v^\prime) F_4 (v^\prime_*) - F_1(v) F_2(v_*) \right]  \big( \psi_3(v^\prime) + \psi_4(v^\prime_*) - \psi_1(v) - \psi_2(v_*) \big) \dd v_* \dd v \dd \sigma,
\end{multline}
where $v^\prime = v_{12}^{34}$ and $v^\prime_* = v_{* 12}^{\ 34}$ are the reference post-collisional velocities of the reaction $\mathcal{S}_1 + \mathcal{S}_2 \leftrightharpoons \mathcal{S}_3 + \mathcal{S}_4$, being the sole needed to characterize the conservation laws associated with $\mathbf{Q}^\CH$. Additionally, for the chemical operator it is also possible to prove \cite{RosSpi} that its components satisfy the following fundamental mass-conservation laws
\begin{equation} \label{eq:mass conservation CH}
\int_{\R^3} Q_1^\CH(\mathbf{F}, \mathbf{F})(v) \dd v = \int_{\R^3} Q_2^\CH(\mathbf{F}, \mathbf{F})(v) \dd v = -\int_{\R^3} Q_3^\CH(\mathbf{F}, \mathbf{F})(v) \dd v = -\int_{\R^3} Q_4^\CH(\mathbf{F}, \mathbf{F})(v) \dd v.
\end{equation}
%
%\smallskip
Combining the weak formulations \eqref{eq:weak formulation EL} and \eqref{eq:weak formulation CH}, one deduces the conservation properties of the multi-species reactive Boltzmann operator $\mathbf{Q}$. More precisely, using the microscopic conservation of momentum and energy satisfied by the elastic collisions \eqref{eq:conservation EL} and by the chemical reactions \eqref{eq:conservation CH}, the conservation properties of $\mathbf{Q}$ are given by
\begin{equation} \label{eq:collisional invariants}
\begin{split}
    & \sum_{i=1}^4 \int_{\R^3} Q_i^\EL(\mathbf{F},\mathbf{F})(v) \psi_i(v) \dd v = 0, \\[2mm]
    & \sum_{i=1}^4 \int_{\R^3} Q_i^\CH(\mathbf{F},\mathbf{F})(v) \psi_i(v) \dd v = 0,
\end{split}
\end{equation}
and these equalities hold simultaneously if and only if $\pmb{\psi}$ is a collision invariant of the reactive mixture, namely 
\begin{equation*}
\pmb{\psi} \in \textrm{Span} \left\{ \mathbf{e}_{13}, \mathbf{e}_{14}, \mathbf{e}_{24}, v_1 \mathbf{m}, v_2 \mathbf{m}, v_3 \mathbf{m}, \frac{1}{2} |v|^2 \mathbf{m} + \mathbf{E} \right\},
\end{equation*}
where we have denoted $\mathbf{e}_{13} = (1, 0, 1, 0)$, $\mathbf{e}_{14} = (1, 0, 0, 1)$, $\mathbf{e}_{24} = (0, 1, 0, 1)$, $\mathbf{m} = (m_1, \ldots, m_4)$ and $\mathbf{E} = (E_1, \ldots, E_4)$. Here, the vectors $\mathbf{e}_{13}$, $\mathbf{e}_{14}$ and $\mathbf{e}_{24}$ prescribe the conservation over time $t \geq 0$ of the partial densities $c_{1, \infty} + c_{3, \infty}$, $c_{1, \infty} + c_{4, \infty}$ and $c_{2, \infty} + c_{4, \infty}$ within the chemical reactions (this property is clearly visible from relations \eqref{eq:mass conservation CH}), where $c_{i, \infty}$ denotes the concentration of the species $\mathcal{S}_i$. Similarly, the collision invariants $v_1 \mathbf{m}$, $v_2 \mathbf{m}$ and $v_3 \mathbf{m}$ give the conservation of total momentum inside the reactive mixture, while $\frac{1}{2} |v|^2 \mathbf{m} + \mathbf{E}$ ensures the conservation of total (kinetic and chemical, or internal) energy. In particular, the Boltzmann equation \eqref{eq:reactive BE 2}, the invariance of mass $m_1 + m_2 = m_3 + m_4$ in the reaction \eqref{eq:reaction} and relations \eqref{eq:collisional invariants} imply that the quantities
\begin{equation} \label{eq:conserved quantities}
\begin{aligned}
c_{\infty} = \sum_{i=1}^{4} \int_{\R^3 \times \T^3} F_i(t,x,v) \dd v \dd x, & \qquad \rho_{\infty} u_{\infty} =  \sum_{i=1}^4 \int_{\R^3 \times \T^3} m_i v F_i(t,x,v) \dd v \dd x
\\[4mm]    \frac{3}{2} K_B \rho_{\infty} T_{\infty} = \sum_{i=1}^4 \int_{\R^3 \times \T^3} & \left( \frac{1}{2} m_i |v - u_{\infty}|^2 + E_i \right) F_i(t,x,v) \dd v \dd x,
\end{aligned}
\end{equation}
are preserved for any time $t \geq 0$. Here, $c_{\infty} = \sum_{i=1}^4 c_{i, \infty}$ stands for the total concentration of the mixture, while $\rho_{\infty} = \sum_{i=1}^4 m_i c_{i,\infty}$ defines the total density of the mixture, $\rho_{\infty}u_{\infty}$ its total momentum and $\frac{3}{2} K_B \rho_{\infty} T_{\infty}$ its total energy.
%\revise{Furthermore, let us point out that the mass-conservation laws \eqref{eq:mass conservation CH}, which are naturally included in \eqref{eq:collisional invariants}, provide a bit more information, as they allow to deduce that all partial sums $c_{1, \infty} + c_{3, \infty}$, $c_{1, \infty} + c_{4, \infty}$ and $c_{2, \infty} + c_{4, \infty}$ are preserved over time $t \geq 0$ by equation \eqref{eq:reactive BE 2}. In particular, the latter are the three conserved quantities associated with the collisional invariants $\mathbf{e}_{13}$, $\mathbf{e}_{14}$ and $\mathbf{e}_{24}$.}

\smallskip
With these global quantities preserved by the kinetic equation, one can deduce an $H$-theorem \cite[Proposition 1]{DesMonSal} for the combination of the operators $\mathbf{Q}^\EL$ and $\mathbf{Q}^\CH$. While the inert component pushes the solution of the Boltzmann equation to relax toward a collisional equilibrium having the shape of a local Maxwellian, the chemical component imposes a further condition binding together the corresponding local concentrations, local temperature, molecular masses and impinging energy of the reaction. Here we are interested in the relaxation toward a global thermodynamic steady state, where the locality in the space variable of the physical quantities associated with the Maxwellian is lost through the action of the transport operator over $\T^3$. The $H$-theorem then implies that the only steady state solution of the equations $\mathbf{Q}(\mathbf{F}, \mathbf{F}) = 0$ and $v \cdot \nabla_x \mathbf{F} = 0$ is the distribution function $\mathbf{F} = (F_1, \ldots, F_4)$, whose components are global Maxwellians of the form
\begin{equation*}
F_i(v) = c_{i, \infty} \left( \frac{m_i}{2 \pi K_B T_\infty} \right)^{3/2} \exp \left( - m_i \frac{|v - u_\infty|^2}{2 K_B T_\infty} \right), \quad v \in \R^3, \ 1 \leq i \leq 4,
\end{equation*}
where the values of the global quantities $(c_{i, \infty})_{1 \leq i \leq 4}$, $u_\infty$ and $T_\infty$ are prescribed by the initial conditions of the Boltzmann equation \eqref{eq:reactive BE 2} via the conservations \eqref{eq:conserved quantities}, and are subject to the additional constraint $\frac{c_{1, \infty} c_{2, \infty}}{c_{3, \infty} c_{4, \infty}} = \left( \frac{m_3 m_4}{m_1 m_2} \right)^{3/2} \exp \left( \frac{E_{12}^{34}}{K_B T_{\infty}} \right)$. In particular, one can perform a suitable translation and dilation of the coordinate system in order to reduce the analysis to the case $u_\infty = 0$ and $K_B T_\infty = 1$. In this way, we are led to introduce the sole global equilibrium of the reactive mixture to be the global Maxwellian distribution $\pmb{\mu} = (\mu_1, \ldots, \mu_4)$ having, for any $1 \leq i \leq 4$ and $v \in \R^3$, the following shape
\begin{equation} \label{eq:global equilibrium}
\mu_i(v) = c_{i, \infty} \left( \frac{m_i}{2 \pi} \right)^{3/2} e^{- m_i \frac{|v|^2}{2}},
\end{equation}
with the global concentrations $(c_{i, \infty})_{1 \leq i \leq 4}$ satisfying the mass action law
\begin{equation} \label{eq:MAL 2}
\frac{c_{1, \infty} c_{2, \infty}}{c_{3, \infty} c_{4, \infty}} = \left( \frac{m_3 m_4}{m_1 m_2} \right)^{3/2} e^{E_{12}^{34}}.
\end{equation}
In what follows, we shall always work in the Hilbert space $L^2(\R^3, \pmb{\mu}^{-1/2})$ weighted by the global equilibrium $\pmb{\mu}^{-1/2} = \big(\mu_1^{-1/2}, \ldots, \mu_4^{-1/2}\big)$.

%%%%%%%%%%%%%%%%%%%    LINEARIZATION    %%%%%%%%%%%%%%%%%%%%
\medskip
\noindent \textbf{Linearized setting.}  Following the methods of the linearized theory to tackle the problem of convergence to equilibrium for the reactive Boltzmann equation, the natural question arising from these considerations is to study the properties of solutions to \eqref{eq:reactive BE 2} in a regime close to the global distribution \eqref{eq:global equilibrium}--\eqref{eq:MAL 2}. To this aim, we write $\mathbf{F} = \pmb{\mu} + \mathbf{f}$ and investigate the behavior of small perturbations $\mathbf{f}$ around the equilibrium $\pmb{\mu}$, solving the perturbed Boltzmann system
\begin{equation} \label{eq:linearized BE}
\partial_t \mathbf{f} + v \cdot \nabla_x \mathbf{f} = \mathbf{L}(\mathbf{f}) + \mathbf{Q}(\mathbf{f}, \mathbf{f}),
\end{equation}
where we have used the equality $\mathbf{Q}(\pmb{\mu}, \pmb{\mu}) = 0$, since by construction $\pmb{\mu}$ ensures the cancellation of both the inert and the chemical collision operators. In this perturbative setting, one expects the first term on the right-hand side to be the dominant one, driving the relaxation of solutions $\mathbf{f}$ toward the equilibrium $\pmb{\mu}$. It is the so-called linearized reactive Boltzmann operator $\mathbf{L} = (L_1, \ldots, L_4)$, given component-wise by
\begin{equation} \label{eq:operator L}
L_i(\mathbf{f}) = L_i^\EL(\mathbf{f}) + L_i^\CH(\mathbf{f}) = \Big[ Q_i^\EL(\pmb{\mu}, \mathbf{f}) + Q_i^\EL(\mathbf{f}, \pmb{\mu}) \Big] + \Big[ Q_i^\CH(\pmb{\mu}, \mathbf{f}) + Q_i^\CH(\mathbf{f}, \pmb{\mu}) \Big], \quad 1 \leq i \leq 4,
\end{equation}
where the elastic part $\mathbf{L}^\EL$, acting on $v \in \R^3$, reads for any $1 \leq i \leq 4$
\begin{equation*}
L_i^\EL(\mathbf{f})(v) = \sum_{j=1}^4 \int_{\R^3\times\Sf} B_{ij}(|v - v_*|, \cos \theta)( f_i^\prime \mu_j^{\prime *} + f_j^{\prime *} \mu_i^\prime -  f_i \mu_j^* - f_j^{\prime *} \mu_i )\dd v_* \dd \sigma,
\end{equation*}
and the linearized chemical operator $\mathbf{L}^\CH$, also acting on $v \in \R^3$, can be written in compact form, for any reaction $\mathcal{S}_i + \mathcal{S}_j \rightarrow \mathcal{S}_h + \mathcal{S}_k$, as
\begin{multline*}
L_i^\CH(\mathbf{f})(v) = \int_{\R^3 \times \Sf} H \left( |v - v_*|^2 - \frac{2 E_{ij}^{hk}}{m_{ij}} \right) B_{ij}^{hk}(|v - v_*|, \cos \theta)
\\[2mm]  \times \left[ \left( \frac{m_{ij}}{m_{hk}} \right)^3 f_h(v^\prime) \mu_k(v^{\prime*}) + \left( \frac{m_{ij}}{m_{hk}} \right)^3 f_k(v^{\prime*}) \mu_h(v^{\prime}) - f_i(v) \mu_j(v_*) - f_j(v_*) \mu_i(v) \right] \dd v_* \dd \sigma.
\end{multline*}
The first step to obtain the convergence to equilibrium in such framework consists in analyzing the spectral properties of $\mathbf{L}$. In particular, one wishes to prove that the first nonzero eigenvalue of the linearized operator is negative and bounded away from zero, allowing to infer (hypo)coercivity estimates for $\mathbf{L}$ which can be used to control the nonlinear operator $\mathbf{Q}$. Results about the kernel and spectrum of the linearized inert operator $\mathbf{L}^\EL$ have already been obtained in recent years \cite{BriDau,BonBouBriGre,DauJunMouZam}. Here, we provide a complementary study on the linearized chemical operator $\mathbf{L}^\CH$ by showing that it acts as a perturbation to the elastic component, in the sense that its associated Dirichlet form can be bounded above by a negative quantity that corrects the spectral gap of $\mathbf{L}^\EL$. The estimates derived in this work are explicit and constitute an essential result to develop a quantitative Cauchy theory of perturbative solutions to \eqref{eq:linearized BE}, as well as to treat the rigorous  derivation of hydrodynamic limits starting from the reactive Boltzmann equation.

%%%%%%%%%%%%%%%%%%%%%%%%%%%%%%%%%%%%%%%%%%%%%%%%%%%%%%%%%%%%
%%%%%%%%%%%%%%%%%%%%%%%%%%%%%%%%%%%%%%%%%%%%%%%%%%%%%%%%%%%%
%%%%%%%%%%%%%%    SECTION 3: MAIN RESULT    %%%%%%%%%%%%%%%%
%%%%%%%%%%%%%%%%%%%%%%%%%%%%%%%%%%%%%%%%%%%%%%%%%%%%%%%%%%%%
%%%%%%%%%%%%%%%%%%%%%%%%%%%%%%%%%%%%%%%%%%%%%%%%%%%%%%%%%%%%

\section{Main result}

\noindent In this section we present the main result of this work, providing a quantitative estimate for the spectral gap of the linearized reactive operator $\mathbf{L}$. In order to state our theorem, we shall first recall well-known structural properties for the kernel of $\mathbf{L}^\EL$ as well as a recent result \cite{BriDau} ensuring the existence of a spectral gap $\lambda_\EL$ for this elastic operator in the space $L^2(\R^3, \pmb{\mu}^{-1/2})$. We proceed with a similar investigation on $\mathbf{L}^\CH$, by determining its kernel and a corresponding basis. This preliminary analysis will provide us with the basic tools needed in the following section, where we will prove that the Dirichlet form associated with the operator $\mathbf{L}^\EL + \mathbf{L}^\CH$ in the space $L^2(\R^3, \pmb{\mu}^{-1/2})$ is upper-bounded by an explicit negative term depending solely on $\lambda_\EL$ and the parameters of the problem.

%%%%%%%%%%%%%%%%%%%%%    NOTATIONS    %%%%%%%%%%%%%%%%%%%%%%
\medskip
\noindent \textbf{Notations.} We begin by collecting here the notations used from now on. For any positive measurable vector-valued function $\mathbf{W} = (W_1, \ldots, W_4) : \R^3 \to (\R_+^*)^4$ in the variable $v$, we define the weighted Hilbert spaces $L^2(\R^3, W_i)$, $1\leq i \leq 4$, by introducing the scalar products and norms
\begin{equation*}
\langle f_i, g_i \rangle_{L_v^2(W_i)} = \int_{\R^3} f_i g_i W_i^2 \dd v, \quad \| f_i \|^2_{L_v^2(W_i)} = \langle f_i, f_i \rangle_{L_v^2(W_i)}, \quad \forall f_i, g_i \in L^2(\R^3, W_i).
\end{equation*}
With this definition, we say that $\mathbf{f} = (f_1, \ldots, f_4) : \R^3 \to \R^4 \in L^2(\R^3, \mathbf{W})$ if $f_i : \R^3 \to \R$ belongs to $L^2(\R^3, W_i)$ and we associate to $L^2(\R^3, \mathbf{W})$ the scalar product and norm
\begin{equation*}
\langle \mathbf{f}, \mathbf{g} \rangle_{L_v^2(\mathbf{W})} = \sum_{i=1}^4 \langle f_i, g_i \rangle_{L_v^2(W_i)}, \quad \| \mathbf{f} \|_{L_v^2(\mathbf{W})} = \left( \sum_{i=1}^4 \| f_i \|^2_{L_v^2(W_i)} \right)^{1/2}.
\end{equation*}
We also recall that our following analysis is performed using the weighted Hilbert space $L^2(\R^3, \pmb{\mu}^{-1/2})$, which is associated with the global Maxwellian weight $\pmb{\mu}^{-1/2} = \big( \mu_1^{-1/2}, \ldots, \mu_4^{-1/2} \big)$. In particular, we will often employ the notation $\mathbf{h} = (h_1, \ldots, h_4)$ to indicate the distribution $\mathbf{f}$ whenever it is rescaled by the weight $\pmb{\mu}^{-1}$, meaning that $h_i := f_i \mu_i^{-1}$ for any $1 \leq i \leq 4$. At last, for the sake of simplicity, we introduce the standard shorthand notation
\begin{equation*}
\langle v \rangle = \Big( 1 + |v|^2 \Big)^{1/2}.
\end{equation*}

%%%%%%%%%%%%%%%    PROPERTIES OF ELASTIC OPERATOR    %%%%%%%%%%%%%%%%
\medskip
\noindent \textbf{Known properties of the elastic component $\mathbf{L}^\EL$.} The linearized Boltzmann multi-species operator $\mathbf{L}^\EL$, relative to the mixture's global equilibrium $\pmb{\mu}$ defined by \eqref{eq:global equilibrium}--\eqref{eq:MAL 2}, is a closed self-adjoint nonpositive operator in $L^2(\R^3, \pmb{\mu}^{-1/2})$, whose Dirichlet form satisfies the relation
\begin{equation} \label{eq:Dirichlet form EL}
\begin{split}
\scalprod{\mathbf{L}^\EL(\mathbf{f}), \mathbf{f}}_{\spaceV} = -\frac{1}{4} \sum_{i,j=1}^4 \int_{\R^6\times\Sf} & B_{ij}(|v-v_*|, \cos \theta) \mu_i(v) \mu_j(v_*) \\[1mm]
& \times \left[ h_i(v^\prime) + h_j(v_*') - h_i(v) - h_j(v_*)\right]^2 \dd v \dd v_* \dd \sigma,
\end{split}
\end{equation}
where we recall that $h_i = f_i \mu_i^{-1}$. The latter is determined by directly applying the weak formulation \eqref{eq:weak formulation EL} of $\mathbf{Q}^\EL$ in linearized form, with $\psi_i = f_i$ for any $1 \leq i \leq 4$.

\smallskip
Moreover, the kernel $\Ker (\mathbf{L}^\EL)$ of the elastic component is spanned by the orthonormal basis $\big( \pmb{\phi}^\EL_k \big)_{1\leq k \leq 8}$ in $L^2(\R^3, \pmb{\mu}^{-1/2})$ given by
\begin{equation} \label{eq:orthonormal basis EL}
\left\{\begin{array}{ll}
\displaystyle \pmb{\phi}^\EL_k = \frac{1}{\sqrt{c_{k,\infty}}} \mu_k \mathbf{e}_k, \quad 1 \leq k \leq 4, \\[2mm]
\displaystyle \pmb{\phi}^\EL_{4+\ell} = \frac{v_\ell}{\left( \sum_{j=1}^4 m_j c_{j,\infty} \right)^{1/2}} \big( m_i \mu_i \big)_{1\leq i\leq 4}, \quad 1\leq \ell \leq 3,\\
\displaystyle \pmb{\phi}^\EL_8 = \frac{1}{\left( \sum_{j=1}^4 c_{j,\infty} \right)^{1/2}} \left( \frac{m_i |v|^2 - 3}{\sqrt{6}} \mu_i \right)_{1\leq i\leq 4}.
\end{array}\right.
\end{equation}
where $(\mathbf{e}_i)_{1\leq i\leq 4}$ denotes the canonical basis of $\R^4$. In particular, for any $\mathbf{f}\in L^2(\R^3, \pmb{\mu}^{-1/2})$, we can define its orthogonal projection onto $\Ker (\mathbf{L}^\EL)$ in $L^2(\R^3, \pmb{\mu}^{-1/2})$ with respect to this basis as
\begin{equation} \label{eq:projection EL}
[\pi^\EL(\mathbf{f})]_i(t,x,v) = \left( n_i(t,x) + m_i u(t,x) \cdot v + \frac{m_i}{2} e(t,x) |v|^2 \right) \mu_i(v), \quad 1 \leq i \leq 4,
\end{equation}
where $n_i \in \R$, $u \in \R^3$ and $e \in \R$ denote the corresponding coordinates. Finally, the elastic operator $\mathbf{L}^\EL$ satisfies the following fundamental property, proved in \cite{BriDau,DauJunMouZam}.

\begin{theorem}[From \cite{BriDau}, Theorem 2.1 (i)] \label{theorem:elastic SG}
Under the assumptions {\normalfont (EL1)--(EL4)}, the linearized elastic operator $\mathbf{L}^\EL$ has an explicit spectral gap, i.e. there exists an explicit constant $\lambda_\EL > 0$ such that, for all $\mathbf{f}\in L^2(\R^3, \pmb{\mu}^{-1/2})$,
\begin{equation} \label{eq:elastic SG}
\big\langle \mathbf{L}^\EL(\mathbf{f}), \mathbf{f} \big\rangle_{\spaceV} \leq -\lambda_\EL \norm{\mathbf{f}-\pi^\EL(\mathbf{f})}_{\spaceVgamma}^2,
\end{equation}
for any $\gamma \in [0,1]$. The constant $\lambda_\EL$ depends only on the different masses $(m_i)_{1\leq i\leq 4}$ and on the collision kernels $(B_{ij})_{1 \leq i, j \leq 4}$.
\end{theorem}

%%%%%%%%%%%%%%    PROPERTIES OF CHEMICAL OPERATOR    %%%%%%%%%%%%%%%%
\medskip
\noindent \textbf{Properties of the linearized chemical operator $\mathbf{L}^\CH$.} We may prove in a very similar way that $\mathbf{L}^\CH$ is a closed, self-adjoint and nonpositive operator in the same space $L^2(\R^3, \pmb{\mu}^{-1/2})$. In particular, the self-adjointness follows by considering the linearized version of the weak formulation \eqref{eq:weak formulation CH} of $\mathbf{Q}^\CH$ with the choice $\psi_i = g_i$, $1 \leq i \leq 4$, and $\mathbf{g} \in L^2(\R^3, \pmb{\mu}^{-1/2})$. This, combined with  the fact that $\bm{\mu}$ is the global chemical equilibrium of the mixture, thus it satisfies the mass action law \eqref{eq:MAL 2} and the relation
\begin{equation*}
\left( \frac{m_1 m_2}{m_3 m_4} \right)^3 \mu_3(v^\prime) \mu_4(v^\prime_*) = \mu_1(v) \mu_2(v_*),
\end{equation*}
allows to show that
\begin{equation*}
\begin{split}
\scalprod{\mathbf{L}^\CH(\mathbf{f}), \mathbf{g}}_{\spaceV} & = - \int_{\R^6\times\Sf} H\left( |v-v_*|^2 - \frac{2 E}{m_{12}} \right) B_{12}^{34}(|v-v_*|, \cos \theta) \mu_1 \mu_2^*  \\[2mm]
& \hspace*{2cm} \times \left[ \frac{f_3'}{\mu'_3} + \frac{f^{' *}_4}{\mu^{' *}_4} - \frac{f_1}{\mu_1} - \frac{f^*_2}{\mu^*_2}\right] \left[ \frac{g_3'}{\mu'_3} + \frac{g^{' *}_4}{\mu^{' *}_4} - \frac{g_1}{\mu_1} - \frac{g^*_2}{\mu^*_2}\right] \dd v \dd v_* \dd \sigma \\[6mm]
& = - \int_{\R^6\times\Sf} H\left( |v-v_*|^2 - \frac{2 E}{m_{12}} \right) B_{12}^{34}(|v-v_*|, \cos \theta) \left[ \frac{f_3'}{\mu'_3} + \frac{f^{' *}_4}{\mu^{' *}_4} - \frac{f_1}{\mu_1} - \frac{f^*_2}{\mu^*_2}\right]  \\[2mm]
& \hspace*{2cm} \times \left[ \left( \frac{m_{12}}{m_{34}} \right)^3 g_3' \mu^{'*}_4 + \left( \frac{m_{12}}{m_{34}} \right)^3 g^{' *}_4 \mu'_3 - g_1 \mu^*_2 - g^*_2 \mu_1 \right] \dd v \dd v_* \dd \sigma, \\[8mm]
& = \scalprod{ \mathbf{f}, \mathbf{L}^\CH(\mathbf{g})}_{\spaceV}.
\end{split} 
\end{equation*}
From the relation above, one also recovers the nonpositivity of $\mathbf{L}^\CH$ by taking $\mathbf{g} = \mathbf{f}$, so that
\begin{multline} \label{eq:Dirichlet form CH}
\scalprod{\mathbf{L}^\CH(\mathbf{f}), \mathbf{f}}_{\spaceV} = - \int_{\R^6\times\Sf} H\left( |v-v_*|^2 - \frac{2 E_{12}^{34}}{m_{12}} \right) B_{12}^{34}(|v-v_*|, \cos \theta) \mu_1(v) \mu_2(v_*)
\\[4mm]   \times \left[ h_3(v^\prime) + h_4(v^\prime_*) - h_1(v) - h_2(v_*)\right]^2 \dd v \dd v_* \dd \sigma,
\end{multline}
where again $h_i = f_i \mu_i^{-1}$. Thanks to \eqref{eq:Dirichlet form EL}--\eqref{eq:Dirichlet form CH}, we can finally characterize the kernel $\Ker(\mathbf{L}^\EL + \mathbf{L}^\CH)$ by finding all solutions $\mathbf{f}\in L^2(\R^3, \pmb{\mu}^{-1/2})$ of the equation
\begin{equation*}
\scalprod{\mathbf{L}^\EL(\mathbf{f}), \mathbf{f}}_{\spaceV} + \scalprod{\mathbf{L}^\CH(\mathbf{f}), \mathbf{f}}_{\spaceV} = 0.
\end{equation*}
These are given by all functions satisfying the coupled relations
\begin{align}
& h_i(v) + h_j(v_*) = h_i(v^\prime) + h_j(v_*'), \label{eq:Cauchy EL} \\[2mm]
& h_1(v) + h_2(v_*) = h_3(v^\prime) + h_4(v^\prime_*), \quad (v^\prime = v_{12}^{34},\ v^\prime_* = v_{*12}^{\ 34}) \label{eq:Cauchy CH}
\end{align}
where $1 \leq i,j \leq 4$ and the velocities are taken in the admissible set satisfying microscopic elastic and reactive conservation laws for momentum and energy. It is well-known \cite{DauJunMouZam} that the solutions to the first Cauchy equation \eqref{eq:Cauchy EL} are all functions of the form
\begin{equation*}
h_i(v) = n_i + m_i u \cdot v + m_i e \frac{|v|^2}{2},
\end{equation*} 
with $n_i$, $e\in\R$, $u\in\R^3$. Now, injecting the latter into the second equation \eqref{eq:Cauchy CH}, from the reactive microscopic conservation laws we deduce a further condition relating the coefficients $(n_i)_{1\leq i\leq 4}$ as
\begin{equation} \label{eq:manifold}
n_4 + n_3 - n_2 - n_1 = e E_{12}^{34}.
\end{equation}
Therefore, we can express the projection $\pi$ of any $\mathbf{f}\in L^2(\R^3, \pmb{\mu}^{-1/2})$ onto the kernel of $\mathbf{L} = \mathbf{L}^\EL + \mathbf{L}^\CH$ in the general form
\begin{equation*} %\label{eq:projection CH}
[\pi(\mathbf{f})]_i (t,x,v) = \left( n_i(t,x) + m_i u(t,x) \cdot v + m_i e(t,x) \frac{|v|^2}{2} \right) \mu_i, \quad 1\leq i \leq 4,
\end{equation*}
where in this case the coordinates $(n_i)_{1\leq i\leq 4}$ and $e$ satisfy the additional relation \eqref{eq:manifold} and $\Ker (\mathbf{L})$ is a 7-dimensional space spanned by the (linearly independent, but non-orthogonal) vectors
\begin{equation} \label{eq:basis CH}
\begin{split}
\pmb{\phi}_1 = \left( \begin{array}{c} \mu_1\\[3mm] 0 \\[3mm] \mu_3 \\[3mm] 0 \end{array} \right), &\qquad \pmb{\phi}_2 = \left( \begin{array}{c} 0 \\[3mm] \mu_2 \\[3mm] 0 \\[3mm] \mu_4 \end{array} \right), \qquad \pmb{\phi}_3 = \left( \begin{array}{c} \mu_1\\[3mm] 0 \\[3mm] 0 \\[3mm] \mu_4 \end{array} \right), \\[4mm] 
\pmb{\phi}_{\ell + 3} = v_\ell \left( \begin{array}{c} m_1 \mu_1 \\[3mm] m_2 \mu_2 \\[3mm] m_3 \mu_3 \\[3mm] m_4 \mu_4 \end{array} \right), & \ \ (\ell = 1, 2, 3), \qquad \pmb{\phi}_7 = \left( \begin{array}{c} \left( m_1 \frac{|v|^2}{2} + E_1\right) \mu_1 \\[3mm] \left( m_2 \frac{|v|^2}{2} + E_2\right) \mu_2 \\[3mm] \left( m_3 \frac{|v|^2}{2} + E_3\right) \mu_3 \\[3mm] \left( m_4 \frac{|v|^2}{2} + E_4\right) \mu_4 \end{array} \right).
\end{split}
\end{equation}

%%%%%%%%%%%%%%%%    MAIN RESULT AND STRATEGY    %%%%%%%%%%%%%%%%%%%
\medskip
\noindent \textbf{Statement of the result and strategy of our proof.} Given the close-to-equilibrium setting introduced in the previous section and encoded by equation \eqref{eq:linearized BE}, we provide a quantitative estimate in the natural space $L^2(\R^3, \pmb{\mu}^{-1/2})$ for the spectral gap of the linearized Boltzmann operator $\mathbf{L}$ modeling elastic and inelastic (chemical) collisions inside a multicomponent gas, thus extending Theorem \ref{theorem:elastic SG} from the case of inert multi-species gases to that of reactive mixtures. In particular, we are able to prove the following result.

\begin{theorem} \label{theorem:reactive SG}
Under the assumptions {\normalfont (EL1)--(EL4)} and {\normalfont (CH1)--(CH4)}, the linearized reactive operator $\mathbf{L} = \mathbf{L}^\EL + \mathbf{L}^\CH$ possesses an explicit spectral gap, i.e. there exists an explicitly computable constant $\lambda > 0$ such that, for all $\mathbf{f}\in L^2(\R^3, \pmb{\mu}^{-1/2})$,
\begin{equation} \label{eq:reactive SG}
\big\langle \mathbf{L}(\mathbf{f}), \mathbf{f} \big\rangle_{\spaceV} \leq -\lambda \norm{\mathbf{f}-\pi(\mathbf{f})}_{\spaceVgamma}^2,
\end{equation}
for any $\gamma \in [0,1]$. In particular, the constant $\lambda$ depends only on $\gamma \in [0,1]$, the total concentration $c_{\infty}$ of the mixture, the different masses $(m_i)_{1 \leq i \leq 4}$, the impinging energy $E_{12}^{34}$, the elastic collision kernels $(B_{ij})_{1 \leq i,j \leq 4}$ and the reactive collision kernels $B_{12}^{34}$ and $B_{34}^{12}$.
\end{theorem}

A possible way to prove inequality \eqref{eq:reactive SG} is to consider a standard splitting \cite{CerIllPul,Vil} of linearized Boltzmann-like operators with cutoff collision kernels into the sum of two operators $\mathbf{K} - \bm{\nu}$, one so-called collision frequency $\bm{\nu}$ that acts as a multiplication on $\mathbf{f}\in L^2(\R^3, \pmb{\mu}^{-1/2})$ and is coercive in this space, and a remainder $\mathbf{K} = \mathbf{L} + \bm{\nu}$ that can be proved to be compact in $L^2(\R^3, \pmb{\mu}^{-1/2})$. The main idea behind this strategy is thus to study the essential spectrum of $\bm{\nu}$, which is easy to determine, and then exploit Weyl's theorem on compact perturbations \cite{Kat} to ensure that its spectral properties are inherited by the linearized Boltzmann operator, in particular its coercivity. Compactness of the operator $\mathbf{K}$ has been investigated in a variety of settings, ranging from single-species gases \cite{Gra1,Uka}, to non-reactive mixtures composed of monatomic \cite{BouGrePavSal,DauJunMouZam} or polyatomic species \cite{Ber1,Ber2,Ber3,Ber4,BorBouSal,BruShaThi}, to multi-species gases involving chemical reactions \cite{Ber5,CarPolSoa}. Such results are particularly relevant to identify the kernel-like representation \cite{Car} of the linearized operator. However, while they guarantee the existence of a spectral gap, they yield no information on its value, and quantifying $\lambda$ in terms of the physical parameters of the problem is crucial to determine explicit rates of convergence to equilibrium \cite{BriDau, DauJunMouZam, Mou2, MouNeu}. Our result provides this missing information.

\smallskip
We prove Theorem \ref{theorem:reactive SG} by adapting the arguments from \cite{BriDau, DauJunMouZam} on the linearized Boltzmann operator for inert mixtures, where it is exploited the fact that $\mathbf{L}^\EL$ can be split into the sum $\mathbf{L}^\EL_\mono + \mathbf{L}^\EL_\cross$ of two operators (respectively modeling elastic collisions between the same species or cross interactions between different species) such that $\mathbf{L}^\EL_\mono$ has an explicit spectral gap \cite{BarMou, Mou1} and $\mathbf{L}^\EL_\cross$ is nonpositive. The reactive linearized Boltzmann operator $\mathbf{L}$ shares a similar structure: when computing its Dirichlet form, one has that $\scalprod{\mathbf{L}^\EL(\mathbf{f}),\mathbf{f}}_{\spaceV}$ satisfies inequality \eqref{eq:elastic SG} with an explicit $\lambda_\EL$ and we have seen that $\scalprod{\mathbf{L}^\CH(\mathbf{f}),\mathbf{f}}_{\spaceV} \leq 0$. This negativity can be quantified in terms of an energy functional $\mathcal{E}(\mathbf{f})$, when looking at distributions that belong to $\Ker(\mathbf{L}^\EL)$. The latter are local equilibria for the linearized elastic operator that take the form \eqref{eq:projection EL} and are characterized by a different coordinate $n_i(t,x)$ for each species. Then, the functional $\mathcal{E}(\mathbf{f})$ controls the relaxation of these collisional equilibria toward distributions of $\Ker (\mathbf{L}^\EL + \mathbf{L}^\CH)$, which are chemical local equilibria whose $n_i(t,x)$ are now linked together with $e(t,x)$ by the additional relation \eqref{eq:manifold} imposed by the mass action law \eqref{eq:MAL 2}. In particular, one can prove that the Dirichlet form of $\mathbf{L}^\CH$ splits into the sum of a positive (but small) term depending on $\norm{\mathbf{f}-\pi^\EL(\mathbf{f})}_{\spaceVgamma}^2$ which can be controlled using $\lambda_\EL$ and a nonpositive remainder depending on $\norm{\pi^\EL(\mathbf{f})}_{\spaceV}^2$ that can be bounded from above by $-\mathcal{E}(\pi^\EL(\mathbf{f}))$. The latter is then estimated in terms of the orthogonal projection of $\mathbf{f}$ onto $\Ker (\mathbf{L}^\EL + \mathbf{L}^\CH)^\perp$ to show that $-\mathcal{E}(\mathbf{f}) \approx - \norm{\mathbf{f} - \pi(\mathbf{f})}_{\spaceVgamma}^2$, allowing to derive inequality \eqref{eq:reactive SG} with explicit bounds on the spectral gap $\lambda$.

\begin{remark} \label{remark3}
    Hypotheses (EL3) and (CH3) on the kinetic part of the elastic $B_{ij}$ and chemical $B_{ij}^{hk}$ collision kernels ensure (see Lemma \ref{lemma:bounds on nu}) that both spaces
    \begin{equation*}
        \begin{split}
            L_v^2\left( \bm{\nu}^\EL \bm{\mu}^{-\frac{1}{2}} \right) &= \left\{ \mathbf{f}: \R^3\to\R^4\ \textrm{measurable s.t. } \sum_{i=1}^4 \int_{\R^3} f_i^2(v) \nu_i^\EL(v) \mu_i^{-1/2}\dd v < +\infty \right\}, \\[4mm]
            L_v^2\left( \bm{\nu}^\CH \bm{\mu}^{-\frac{1}{2}} \right) &= \left\{ \mathbf{f}: \R^3\to\R^4\ \textrm{measurable s.t. } \sum_{i=1}^4 \int_{\R^3} f_i^2(v) \nu_i^\CH(v) \mu_i^{-1/2}\dd v < +\infty \right\},
        \end{split}
    \end{equation*}
    are equivalent to $\spaceVgamma$, where the elastic $\bm{\nu}^\EL$ and chemical $\bm{\nu}^\CH$ collision frequencies are defined respectively in \eqref{eq:nu EL} and \eqref{eq:nu CH}. Replacing these conditions with (EL3')--(EL5) and (CH3')--(CH5)--(CH6) respectively, in order to consider more general collision kernels (see the discussion in Remarks \ref{remark1} and \ref{remark2}), would prevent this equivalence between spaces. Therefore, in this case the spectral gap estimate \eqref{eq:elastic SG} for the linearized elastic operator $\mathbf{L}^\EL$, given by Theorem \ref{theorem:elastic SG}, would read
    \begin{equation*}
        \big\langle \mathbf{L}^\EL(\mathbf{f}), \mathbf{f} \big\rangle_{\spaceV} \leq -\lambda_\EL \norm{\mathbf{f}-\pi^\EL(\mathbf{f})}_{L_v^2\left(\bm{\nu}^\EL \bm{\mu}^{-\frac{1}{2}}\right)}^2, \quad \forall \mathbf{f}\in \spaceV,
    \end{equation*}
    while the spectral gap estimate \eqref{eq:reactive SG} for the full linearized reactive operator $\mathbf{L}$, given by Theorem \ref{theorem:reactive SG}, should be replaced with
    \begin{equation*}
        \big\langle \mathbf{L}(\mathbf{f}), \mathbf{f} \big\rangle_{\spaceV} \leq -\lambda \norm{\mathbf{f}-\pi(\mathbf{f})}_{L_v^2\left((\bm{\nu}^\EL + \bm{\nu}^\CH) \bm{\mu}^{-\frac{1}{2}}\right)}^2, \quad \forall \mathbf{f}\in \spaceV.
    \end{equation*}
\end{remark}

%%%%%%%%%%%%%%%%%%%%%%%%%%%%%%%%%%%%%%%%%%%%%%%%%%%%%%%%%%%%
%%%%%%%%%%%%%%%%%%%%%%%%%%%%%%%%%%%%%%%%%%%%%%%%%%%%%%%%%%%%
%%%%%%%%    SECTION 4: SPECTRAL GAP ESTIMATES    %%%%%%%%%%%
%%%%%%%%%%%%%%%%%%%%%%%%%%%%%%%%%%%%%%%%%%%%%%%%%%%%%%%%%%%%
%%%%%%%%%%%%%%%%%%%%%%%%%%%%%%%%%%%%%%%%%%%%%%%%%%%%%%%%%%%%
	
\section{Proof of Theorem \ref{theorem:reactive SG}}

%%%%%%%%%%%%%%%%    UPPER/LOWER BOUNDS ON NU    %%%%%%%%%%%%%%%%%%
\noindent \textbf{Preliminaries on the chemical collision frequencies.} Our first aim is to prove that the chemical collision frequency $\bm{\nu}^\CH$, acting as a multiplicative operator on $\mathbf{f} \in L^2(\R^3, \pmb{\mu}^{-1/2})$ and defined componentwise by
\begin{equation} \label{eq:nu CH}
\nu_i^\CH(v) = \int_{\R^3 \times \Sf} H\left( |v-v_*|^2 - \frac{2 E_{ij}^{hk}}{m_{ij}} \right) B_{ij}^{hk}(|v-v_*|, \cos \theta) \mu_j(v_*) \dd v_* \dd \sigma, \quad v \in \R^3,
\end{equation}
for any reaction $\mathcal{S}_i + \mathcal{S}_j \rightarrow \mathcal{S}_h + \mathcal{S}_k$, can be bounded above and below by the function $\scalprod{v}^\gamma$. In fact, since the latter appears in the weighted $L^2$ norm from the spectral gap estimate \eqref{eq:elastic SG} satisfied by the elastic operator, it is clear that in order to compare the Dirichlet form of $\mathbf{L}^\CH$ with that of $\mathbf{L}^\EL$, one needs to determine suitable relative bounds between the weighted norms bounding these two operators (or rather, their respective collision frequencies). In particular, we obtain the following equivalence result between $\nu_i^\CH(v)$ and $\scalprod{v}^\gamma$.

\begin{lemma} \label{lemma:bounds on nu}
Given the chemical collision frequency $\bm{\nu}^\CH$ defined by \eqref{eq:nu CH}, for any $1 \leq i \leq 4$ there exist explicitly computable constants $\underline{\nu}_i$, $\overline{\nu}_i > 0$ such that
\begin{equation} \label{eq:bounds on nu}
\underline{\nu}_i \scalprod{v}^\gamma \leq \nu_i^\CH(v) \leq \overline{\nu}_i \scalprod{v}^\gamma, \quad \forall v \in \R^3.
\end{equation}
In particular, the constants $\underline{\nu}_i$ and $\overline{\nu}_i$ depend solely on $\gamma \in [0,1]$, the total concentration $c_\infty$ of the mixture, the different masses $(m_i)_{1 \leq i \leq 4}$, the impinging energy $E_{12}^{34}$ and the reactive collision kernels $B_{12}^{34}$ and $B_{34}^{12}$.
\end{lemma}

\begin{proof}
It is well known \cite{CerIllPul, Vil} that $\scalprod{v}^\gamma$ is equivalent to $1 + |v|^\gamma$. Therefore, it is enough to prove the following upper and lower bounds:
\begin{equation*}
\underline{\nu}_i (1 + |v|^\gamma) \leq \nu_i^\CH(v) \leq \overline{\nu}_i (1 + |v|^\gamma), \quad \forall v \in \R^3.
\end{equation*}
Let us reduce our study to the first and third species, $\mathcal{S}_1$ and $\mathcal{S}_3$, since the others can be treated in a similar way. In particular, we determine as example the explicit form of $\overline{\nu}_1$ and $\underline{\nu}_1$. For simplicity and convenience, across all estimates we shall also successively redefine the positive constants $\underline{\nu}_i$ and $\overline{\nu}_i$. Recall the elementary inequalities $(a-b)^2 \geq \frac{1}{2} a^2 - b^2$ and $(a+b)^\alpha \leq a^\alpha + b^\alpha$ for any $\alpha\in [0,1]$, which will be used with $\alpha = \gamma$ for the upper bound and $\alpha=\frac{1-\gamma}{2}$ for the lower bound. For $\mathcal{S}_1$, we can estimate $\nu_1^\CH$ from above as
\begin{equation*}
\begin{split}
\nu_1^\CH(v) & = \int_{\R^3\times\Sf} H\left( |v-v_*|^2 - \frac{2 E_{12}^{34}}{m_{12}} \right) B_{12}^{34}(|v-v_*|, \cos \theta) \mu_2(v_*) \dd v_* \dd \sigma \\[6mm]
& = C_{12}^{34} \int_{\Sf} b_{12}^{34} (\cos \theta) \dd \sigma \int_{\R^3} H\bigg( |v-v_*|^2 - \frac{2 E_{12}^{34}}{m_{12}} \bigg) \\[3mm]
& \hspace*{5cm} \times \left[ \frac{m_{12}}{m_{34}} \left( |v-v_*|^2 - \frac{2 E_{12}^{34}}{m_{12}} \right) \right]^{\frac{\gamma+1}{4}} |v-v_*|^{\frac{\gamma-1}{2}} \mu_2(v_*) \dd v_* \\[4mm]
& \leq \overline{\nu}_1 \int_{\R^3} |v-v_*|^{\frac{\gamma+1}{2}} |v-v_*|^{\frac{\gamma-1}{2}} \mu_2(v_*) \dd v_* \\[4mm]
%& = C \int_{\R^3} |v-v_*|^\gamma \mu_2(v_*) \dd v_* \\[4mm]
& \leq \overline{\nu}_1 \int_{\R^3} (|v|^\gamma + |v_*|^\gamma) \mu_2(v_*) \dd v_* \\[6mm]
& \leq \overline{\nu}_1 (1 + |v|^\gamma),
\end{split}
\end{equation*}
by setting $\overline{\nu}_1$ to the value
\begin{equation*}
\overline{\nu}_1 = c_\infty C_{12}^{34} \left( \frac{m_{12}}{m_{34}} \right)^{\frac{\gamma+1}{4}} \max \left\{ 1,\ \frac{4}{\sqrt{\pi}} \left( \frac{2}{m_2} \right)^{\frac{\gamma}{2}} \overline{\Gamma}(\gamma+3,0) \right\} \int_{\Sf} b_{12}^{34} (\cos \theta) \dd \sigma,
\end{equation*}
with $\overline{\Gamma}(s,y) = \int_y^{+\infty} r^{s-1} e^{-r} \dd r$ denoting the upper incomplete Gamma function, and the constant is positive and bounded thanks to our assumption (CH4) on the angular part of the reactive collision kernel $B_{12}^{34}$.

\smallskip
For the lower bound the estimates should be more careful. We distinguish two cases. For $|v|^2 \geq 2\left( 1 + \frac{2 E_{12}^{34}}{m_{12}} \right) =: \bar{C}$, we successively get
\begin{equation*}
\begin{split}
\nu_1^\CH(v) & \geq \underline{\nu}_1 \int_{\R^3} H\left( |v-v_*|^2 - \frac{2 E_{12}^{34}}{m_{12}} \right) \frac{\left( |v-v_*|^2 - \frac{2 E_{12}^{34}}{m_{12}} \right)^{\frac{\gamma+1}{4}}}{|v-v_*|^{\frac{1-\gamma}{2}}}\mu_2(v_*) \dd v_* \\[4mm]
& \geq \underline{\nu}_1 \int_{\R^3} H\left( \frac{1}{2} |v|^2 - |v_*|^2 - \frac{2 E_{12}^{34}}{m_{12}} \right) \frac{\left( \frac{1}{2} |v|^2 - |v_*|^2 - \frac{2 E_{12}^{34}}{m_{12}} \right)^{\frac{\gamma+1}{4}}}{|v|^{\frac{1-\gamma}{2}} + |v_*|^{\frac{1-\gamma}{2}}} \mu_2(v_*) \dd v_* \\[4mm]
& \geq \underline{\nu}_1 \int_{|v_*|^2 \leq \frac{1}{2}} H\left( \frac{1}{4 \bar{C}} |v|^2 \right) \frac{\left( \frac{1}{4 \bar{C}} |v|^2 \right)^{\frac{\gamma+1}{4}}}{2 |v|^{\frac{1-\gamma}{2}}} \mu_2(v_*) \dd v_* \\[6mm]
& \geq \underline{\nu}_1 |v|^\gamma \geq \ \underline{\nu}_1 (1 + |v|^\gamma),
\end{split}
\end{equation*}
where we find
\begin{equation*}
\underline{\nu}_1 = \frac{c_\infty}{2 \sqrt{\pi}} C_{12}^{34} \left( \frac{1}{4 \bar{C}} \frac{m_{12}}{m_{34}} \right)^{\frac{\gamma+1}{4}} \underline{\Gamma}\left( \frac{3}{2}, \frac{m_2}{4} \right) \int_{\Sf} b_{12}^{34} (\cos \theta) \dd \sigma,
\end{equation*}
with $\underline{\Gamma}(s,y) = \int_0^y r^{s-1} e^{-r} \dd r$ denoting the lower incomplete Gamma function, and $\underline{\nu}_1$ is positive thanks again to hypothesis (CH4).

\smallskip
For $|v|^2 \leq 2\left( 1 + \frac{2 E_{12}^{34}}{m_{12}} \right)$, we instead obtain
\begin{equation*}
\begin{split}
\nu_1^\CH(v) & \geq \underline{\nu}_1 \int_{\R^3} H\left( |v-v_*|^2 - \frac{2 E_{12}^{34}}{m_{12}} \right)  \frac{\left( |v-v_*|^2 - \frac{2 E_{12}^{34}}{m_{12}} \right)^{\frac{\gamma+1}{4}}}{|v-v_*|^{\frac{1-\gamma}{2}}} \mu_2(v_*) \dd v_* \\[2mm]
& \geq \underline{\nu}_1 \int_{\R^3} H\left( \frac{1}{2} |v_*|^2 - |v|^2 - \frac{2 E_{12}^{34}}{m_{12}} \right) \frac{\left( \frac{1}{2} |v_*|^2 - |v|^2 - \frac{2 E_{12}^{34}}{m_{12}} \right)^{\frac{\gamma+1}{4}}}{|v|^{\frac{1-\gamma}{2}} + |v_*|^{\frac{1-\gamma}{2}}} \mu_2(v_*) \dd v_* \\[2mm]
& \geq \underline{\nu}_1 \int_{|v_*|^2 \geq 3 \bar{C}} H\left( \frac{1}{2} |v_*|^2 - 2 - \frac{6 E_{12}^{34}}{m_{12}} \right) \frac{\left( \frac{1}{2} |v_*|^2 - 2 - \frac{6 E_{12}^{34}}{m_{12}} \right)^{\frac{\gamma+1}{4}}}{2 |v_*|^{\frac{1-\gamma}{2}}} \mu_2(v_*) \dd v_* \\[6mm]
& \geq \underline{\nu}_1 \geq \ \underline{\nu}_1 (1 + |v|^\gamma),
\end{split}
\end{equation*}
where now we compute
\begin{equation*}
\underline{\nu}_1 = \frac{c_\infty}{\sqrt{\pi} \left( 1 + \bar{C}^{\frac{\gamma}{2}} \right)} C_{12}^{34} \left( \frac{m_{12}}{m_{34}} \right)^{\frac{\gamma+1}{4}} \left( \frac{2}{m_2} \right)^{\frac{\gamma-1}{4}} \overline{\Gamma}\left( \frac{\gamma + 5}{4}, \frac{3}{2} m_2 \bar{C} \right) \int_{\Sf} b_{12}^{34} (\cos \theta) \dd \sigma.
\end{equation*}
To conclude the estimate, one can simply take $\underline{\nu}_1$ to be the minimum between this constant and the one obtained from the previous lower bound. 

\smallskip
Turning next to $\mathcal{S}_3$, we start with the lower bound which is easier to recover in this case. We get
\begin{equation*}
\begin{split}
\nu_3^\CH(v) & =  \int_{\R^3\times\Sf} B_{34}^{12}(|v-v_*|, \cos \theta) \mu_4(v_*) \dd v_* \dd \sigma \\[4mm]
& =  C_{34}^{12} \int_{\Sf} b_{34}^{12} (\cos \theta) \dd \sigma \int_{\R^3} \left[ \frac{m_{34}}{m_{12}} \left( |v-v_*|^2 + \frac{2 E_{12}^{34}}{m_{34}} \right) \right]^{\frac{\gamma+1}{4}} |v-v_*|^{\frac{\gamma-1}{2}} \mu_4(v_*) \dd v_* \\[4mm]
& \geq \underline{\nu}_3 \int_{\R^3} |v-v_*|^{\frac{\gamma+1}{2}} |v-v_*|^{\frac{\gamma-1}{2}} \mu_4(v_*) \dd v_* \\[6mm]
& \geq \underline{\nu}_3 (1 + |v|^\gamma).
\end{split}
\end{equation*}
Al last, for its upper bound we successively compute
\begin{equation*}
\begin{split}
\nu_3^\CH(v) & \leq \overline{\nu}_3 \int_{\R^3} \frac{\left( |v-v_*|^2 + \frac{2 E_{12}^{34}}{m_{34}} \right)^{\frac{\gamma+1}{4}}}{|v-v_*|^{\frac{1-\gamma}{2}}}\mu_4(v_*) \dd v_* \\[4mm]
& \leq \overline{\nu}_3 \int_{\R^3} \frac{|v-v_*|^{\frac{\gamma+1}{2}} + 1}{|v-v_*|^{\frac{1-\gamma}{2}}}\mu_4(v_*) \dd v_* \\[4mm]
& \leq \overline{\nu}_3 \left[ (1+|v|^\gamma) + \int_{\R^3} |v-v_*|^{\frac{\gamma-1}{2}} \exp\left( -\frac{m_4}{2} |v_*|^2 \right) \dd v_* \right].
\end{split}
\end{equation*}
It only remains to prove that the second term can be bounded by the first one. We notice that the singular part $|v|^{\frac{\gamma-1}{2}}$ can be decomposed into two parts as $|v|^{\frac{\gamma-1}{2}} \chi_{\{|v| < 1\}} + |v|^{\frac{\gamma-1}{2}}\chi_{\{|v| \geq 1\}}$  that belong respectively to $L^1_v(\R^3)$ and $L^\infty_v(\R^3)$. Therefore, convolution with the Maxwellian ensures that the second term actually belongs to $L^\infty_v(\R^3)$ and can thus be bounded by $1+|v|^\gamma$.
\end{proof}

%%%%%%%%%%%%%%%    PROOF OF MAIN RESULT - STEP 1    %%%%%%%%%%%%%%%%

\medskip
To complete the proof of Theorem \ref{theorem:reactive SG}, we proceed via several steps. For convenience of the reader, let us sketch the strategy. In \textbf{Step 1} we split $\mathbf{L}$ into elastic and chemical parts to initially recover
\begin{equation*}
    \scalprod{\mathbf{L}(\mathbf{f}), \mathbf{f}}_{\spaceV}  \le \scalprod{\mathbf{L}^{\EL}(\mathbf{f}), \mathbf{f}}_{\spaceV} + \eta \scalprod{\mathbf{L}^{\CH}(\mathbf{f}), \mathbf{f}}_{\spaceV},
\end{equation*}
for any $\eta\in (0,1]$. From Theorem \ref{theorem:elastic SG}, the first part is controlled  by $-\lambda_{\EL}\norm{\mathbf{f} - \pi^{\EL}(\mathbf{f})}_{\spaceVgamma}^2$. To estimate the second part, in \textbf{Step 2}, we split $\mathbf{f} = \pi^{\EL}(\mathbf{f}) + \big(\mathbf{f} - \pi^{\EL}(\mathbf{f})\big)$ and show, by using the precise formula of $\mathbf{L}^{\CH}$, that the orthogonal part $\mathbf{f} - \pi^{\EL}(\mathbf{f})$ contained in the Dirichlet form $\scalprod{\mathbf{L}^{\CH}(\mathbf{f}), \mathbf{f}}_{\spaceV}$ can be bounded by $\norm{\mathbf{f} - \pi^{\EL}(\mathbf{f})}_{\spaceVgamma}^2$. For the remaining term that appears in $\scalprod{\mathbf{L}^{\CH}(\mathbf{f}), \mathbf{f}}_{\spaceV}$ and depends on $\pi^{\EL}(\mathbf{f})$, we show in \textbf{Step 3} that it is controlled by the positive functional $\mathcal{E}(\mathbf{g}):= \left(n_4 + n_3 - n_2 - n_1 - eE_{12}^{34}\right)^2$ for any $\mathbf{g} \in \Ker(\mathbf{L}^{\EL})\cap \text{Dom}(\mathbf{L}^{\CH})$, where $(n_i)_{1 \leq i \leq 4}$ and $e$ denote the coordinates corresponding to species densities and energy of $\mathbf{g}$ in the orthonormal basis $(\pmb{\phi}_k^{\EL})$ given in \eqref{eq:orthonormal basis EL} and \eqref{eq:projection EL}. In order to conclude, we are left with determining a relation between the three quantities $\norm{\mathbf{f} - \pi(\mathbf{f})}_{\spaceVgamma}^2$, $\norm{\mathbf{f} - \pi^{\EL}(\mathbf{f})}_{\spaceVgamma}^2$ and $\mathcal{E}(\pi^{\EL}(\mathbf{f}))$. This is done in \textbf{Step 4}, where we prove that for some explicit constant $C > 0$ it holds
\begin{equation*}
    -\norm{\mathbf{f} - \pi^{\EL}(\mathbf{f})}_{\spaceVgamma}^2 \le -\frac{1}{2} \norm{\mathbf{f} - \pi(\mathbf{f})}_{\spaceVgamma}^2 + C \mathcal{E}(\pi^{\EL}(\mathbf{f})),
\end{equation*}
by using an explicit orthonormal basis $\big(\pmb{\phi}_k^{\CH}\big)_{1\le k \le 8}$ of $\Ker(\mathbf{L}^{\EL} + \mathbf{L}^{\CH})$, given by \eqref{eq:orthonormal basis CH}, and an arbitrary orthonormal basis $\big(\pmb{\psi}_k^{\EL}\big)_{1\le k\le 8}$ of $\Ker(\mathbf{L}^{\EL})$. Combining all the previous estimates in \textbf{Step 5}, we end up with the desired spectral gap estimate \eqref{eq:reactive SG} of Theorem \ref{theorem:reactive SG}.

\medskip
\noindent \textbf{Step 1 -- Splitting into elastic and chemical components.} We continue by adapting the strategy developed in \cite{DauJunMouZam, BriDau} to prove the estimate \eqref{eq:elastic SG} for $\mathbf{L}^\EL$, which was based on the essential fact that the multi-species elastic operator can be split into a single-species operator (i.e. acting on distributions from the same species) and a component involving all other cross interactions. Similarly, the idea here is to use the definition of $\mathbf{L}$ as the sum $\mathbf{L}^\EL + \mathbf{L}^\CH$ and exploit that $\mathbf{L}^\EL$ has an explicit spectral gap. In particular, letting $\mathbf{f}\in \Dom(\mathbf{L}^\EL)$ and recalling Theorem \ref{theorem:elastic SG}, we initially deduce that
\begin{equation*}
\begin{split}
\scalprod{\mathbf{L}(\mathbf{f}), \mathbf{f}}_{\spaceV} & = \scalprod{\mathbf{L}^\EL(\mathbf{f}), \mathbf{f}}_{\spaceV} + \scalprod{\mathbf{L}^\CH(\mathbf{f}), \mathbf{f}}_{\spaceV}\\[4mm]
& \leq -\lambda_\EL \norm{\mathbf{f} - \pi^\EL(\mathbf{f})}^2_{\spaceVgamma} + \eta \scalprod{\mathbf{L}^\CH(\mathbf{f}), \mathbf{f}}_{\spaceV},
\end{split}
\end{equation*}
for any $\eta \in (0,1]$, since from \eqref{eq:Dirichlet form CH} follows that $(1-\eta) \scalprod{\mathbf{L}^\CH(\mathbf{f}), \mathbf{f}}_{\spaceV} \leq 0$. We now wish to prove that part of the second term can actually be absorbed into the first one depending on $\lambda_\EL$.

%%%%%%%%%%%%%%%%    PROOF OF MAIN RESULT - STEP 2    %%%%%%%%%%%%%%%%
\medskip
\noindent \textbf{Step 2 -- Control of the orthogonal component.} To show this, we use the projection $\pi^\EL$ to split $\mathbf{f} = \pi^\EL(\mathbf{f}) + \big( \mathbf{f}-\pi^\EL(\mathbf{f}) \big)$ into a macroscopic and a kinetic part, this latter being given by the orthogonal component $\mathbf{f}-\pi^\EL(\mathbf{f}) \in \Ker(\mathbf{L}^\EL)^\perp$. Recalling that $h_i = f_i \mu_i^{-1}$ and denoting $\displaystyle \mathcal{A}_{12}^{34}(\mathbf{h}) := h_3(v_{12}^{34}) + h_4(v_{*12}^{\ 34}) - h_1(v) - h_2(v_*)$, we can use the straightforward inequality $\displaystyle \left[ \mathcal{A}_{12}^{34}(\mathbf{h}) \right]^2 \geq \frac{1}{2} \left[ \mathcal{A}_{12}^{34}(\pi^\EL(\mathbf{h})) \right]^2 - \left[ \mathcal{A}_{12}^{34}(\mathbf{h} - \pi^\EL(\mathbf{h})) \right]^2$ to deduce that
\begin{equation*}
\begin{split}
\scalprod{\mathbf{L}(\mathbf{f}), \mathbf{f}}_{\spaceV} & \leq - \lambda_\EL \norm{\mathbf{f} - \pi^\EL(\mathbf{f})}^2_{\spaceVgamma} \\[6mm]
& \hspace*{-2cm} - \frac{\eta}{2} \int_{\R^6\times\Sf} H\left( |v-v_*|^2 - \frac{2 E_{12}^{34}}{m_{12}} \right) B_{12}^{34}(|v-v_*|, \cos \theta) \mu_1(v) \mu_2(v_*) \left[ \mathcal{A}_{12}^{34}(\pi^\EL(\mathbf{f})) \right]^2 \dd v \dd v_* \dd \sigma \\[6mm]
& \hspace*{-2cm} + \eta \int_{\R^6\times\Sf} H\left( |v-v_*|^2 - \frac{2 E_{12}^{34}}{m_{12}} \right) B_{12}^{34}(|v-v_*|, \cos \theta) \mu_1(v) \mu_2(v_*) \left[ \mathcal{A}_{12}^{34}(\mathbf{h} - \pi^\EL(\mathbf{h})) \right]^2 \dd v \dd v_* \dd \sigma.
\end{split}
\end{equation*}
Now, the third term involving the orthogonal component can actually be estimated from above by $\norm{\mathbf{f} - \pi^\EL(\mathbf{f})}^2_{\spaceVgamma}$. Indeed, using the changes of variables $(v, v_*) \mapsto (v_*, v)$, $(v_{12}^{34}, v_{* 12}^{\ 34}) \mapsto (v_*, v)$ and the invariance properties of the collision kernels, specifically their symmetries when interchanging species $(1,2) \leftrightarrow (2,1)$, $(3,4) \leftrightarrow (4,3)$ and their micro-reversibility \eqref{eq:micro-reversibility}, we see that
\begin{equation} \label{eq:C_nu}
\begin{split}
& \int_{\R^6\times\Sf} H \left( |v-v_*|^2 - \frac{2 E_{12}^{34}}{m_{12}} \right) B_{12}^{34}(|v-v_*|, \cos \theta) \mu_1(v) \mu_2(v_*) \left[ A_{12}^{34}(\mathbf{h} - \pi^\EL(\mathbf{h})) \right]^2 \dd v \dd v_* \dd \sigma \\[4mm]
& \leq 4 \int_{\R^6\times\Sf} H\left( |v-v_*|^2 - \frac{2 E_{12}^{34}}{m_{12}} \right) B_{12}^{34}(|v-v_*|, \cos \theta) \mu_1(v) \mu_2(v_*) \\[2mm]
& \hspace{1.5cm} \times \left[ \left(h_3^\prime - \pi^\EL(h_3^\prime)\right)^2 + \left(h_4^{\prime *} - \pi^\EL(h_4^{\prime *})\right)^2 + \left(h_1 - \pi^\EL(h_1)\right)^2 + \left(h_2^* - \pi^\EL(h_2^*)\right)^2 \right] \dd v \dd v_* \dd \sigma \\[4mm]
& \leq 4 \Bigg\{ \int_{\R^6\times\Sf} B_{12}^{34}(|v-v_*|, \cos \theta) \left(f_1 - \pi^\EL(f_1)\right)^2 \mu_1(v)^{-1} \mu_2(v_*) \dd v\dd v_* \dd \sigma \ + \\[2mm]
& \hspace{1.3cm} \int_{\R^6\times\Sf} B_{12}^{34}(|v-v_*|, \cos \theta) \left(f_2 - \pi^\EL(f_2)\right)^2 \mu_2(v)^{-1} \mu_1(v_*) \dd v\dd v_* \dd \sigma \ + \\[2mm]
& \hspace{1.3cm} \int_{\R^6\times\Sf} B_{34}^{12}(|v-v_*|, \cos \theta) \left(f_3 - \pi^\EL(f_3)\right)^2 \mu_3(v)^{-1} \mu_4(v_*) \dd v\dd v_* \dd \sigma \ + \\[2mm]
& \hspace{1.3cm} \int_{\R^6\times\Sf} B_{34}^{12}(|v-v_*|, \cos \theta) \left(f_4 - \pi^\EL(f_4)\right)^2 \mu_4(v)^{-1} \mu_3(v_*) \dd v\dd v_* \dd \sigma \Bigg \} \\[4mm]
& \leq 16 C_{\nu} \norm{\mathbf{f} - \pi^\EL(\mathbf{f})}^2_{\spaceVgamma},
\end{split}
\end{equation}
where we have set $\displaystyle C_{\nu} = \max_{1 \leq i \leq 4} \overline{\nu}_i$ and each $\overline{\nu}_i$ comes from Lemma \ref{eq:bounds on nu}. Thus, clearly,
\begin{equation} \label{eq:main estimate 1}
\begin{split}
\scalprod{\mathbf{L}(\mathbf{f}), \mathbf{f}}_{\spaceV} \leq - (\lambda_{\mathsf{EL}} - 16 C_{\nu} \eta) & \norm{\mathbf{f} - \pi^\EL(\mathbf{f})}^2_{\spaceVgamma} \\[4mm] 
& \hspace*{0.5cm} + \frac{\eta}{2} \scalprod{\mathbf{L}^\CH(\pi^\EL(\mathbf{f})), \pi^\EL(\mathbf{f})}_{\spaceV}.
\end{split}
\end{equation}

%%%%%%%%%%%%%%%    PROOF OF MAIN RESULT - STEP 3    %%%%%%%%%%%%%%%%
\medskip
\noindent \textbf{Step 3 -- Connection with the manifold of chemical equilibria.} The goal is now to prove that for every $\mathbf{f} \in \Ker(\mathbf{L}^\EL) \cap \Dom(\mathbf{L}^\CH)$ the following estimate holds
\begin{equation*}
\scalprod{\mathbf{L}^\CH(\mathbf{f}), \mathbf{f}}_{\spaceV} \leq - C_b \mathcal{E}(\mathbf{f}),
\end{equation*}
for some constant $C_b>0$, where the functional $\mathcal{E}: \Ker(\mathbf{L}^\EL) \to \R_+$ is defined as
\begin{equation*}
\mathcal{E}(\mathbf{f}) = \left( n_4 + n_3 - n_2 - n_1 - e E_{12}^{34} \right)^2,
\end{equation*}
with $(n_i)_{1\leq i\leq 4}$ and $e$ describing the coordinates of $\mathbf{f}\in \Ker(\mathbf{L}^\EL)$ with respect to the orthonormal basis $(\pmb{\phi}^\EL_k)_{1\leq k\leq 8}$ given by \eqref{eq:orthonormal basis EL}. Recalling that in this basis the orthogonal projection $\pi^\EL(\mathbf{f})$ has the form \eqref{eq:projection EL}, we deduce that the quantity $\mathcal{A}_{12}^{34}(\pi^\EL(\mathbf{f}))$ explicitly writes as
\begin{equation*}
A_{12}^{34}(\pi^\EL(\mathbf{f}))^2 = \left( n_4 + n_3 - n_2 - n_1 - e E_{12}^{34} \right)^2,
\end{equation*}
and therefore
\begin{equation} \label{eq:C_b}
\begin{split}
& -\scalprod{\mathbf{L}^\CH(\pi^\EL(\mathbf{f})), \pi^\EL(\mathbf{f})}_{\spaceV} \\[3mm]
& = \int_{\R^6\times\Sf} H\left( |v-v_*|^2 - \frac{2 E_{12}^{34}}{m_{12}} \right) B_{12}^{34}(|v-v_*|, \cos \theta) \mu_1(v) \mu_2(v_*) A_{12}^{34}(\pi^\EL(\mathbf{f}))^2 \dd v \dd v_* \dd \sigma \\[3mm]
& = \left( n_4 + n_3 - n_2 - n_1 - e E_{12}^{34} \right)^2 \int_{\R^6\times\Sf} H\left( |v-v_*|^2 - \frac{2 E_{12}^{34}}{m_{12}} \right) B_{12}^{34}(|v-v_*|, \cos \theta) \mu_1(v) \mu_2(v_*) \dd v \dd v_* \dd \sigma \\[3mm]
& \geq \underline{\nu}_1 \left( n_4 + n_3 - n_2 - n_1 - e E_{12}^{34} \right)^2 \int_{\R^3} (1 + |v|^\gamma) \mu_1(v) \dd v = C_b \mathcal{E}(\pi^\EL(\mathbf{f})),
\end{split}
\end{equation}
where $\displaystyle C_b = \underline{\nu}_1 \int_{\R^3} (1 + |v|^\gamma) \mu_1(v) \dd v$ is obviously positive and $\underline{\nu}_1 > 0$ has been explicitly computed inside the proof of Lemma \ref{eq:bounds on nu}.

\smallskip
Going back to our main estimate \eqref{eq:main estimate 1}, we then recover
\begin{equation} \label{eq:main estimate 2}
\scalprod{\mathbf{L}(\mathbf{f}), \mathbf{f}}_{\spaceV} \leq  - (\lambda_{\mathsf{EL}} - 16 C_{\nu} \eta) \norm{\mathbf{f}-\pi^\EL(\mathbf{f})}^2_{\spaceVgamma} - \frac{\eta C_b}{2} \mathcal{E}(\pi^\EL(\mathbf{f})).
\end{equation}

%%%%%%%%%%%%%%%    PROOF OF MAIN RESULT - STEP 4    %%%%%%%%%%%%%%%%
\medskip
\noindent \textbf{Step 4 -- Estimate on the global concentrations.} To conclude, we need to find a relation between the three quantities $\mathcal{E}(\pi^\EL(\mathbf{f}))$, $\norm{\mathbf{f} - \pi^\EL(\mathbf{f})}$ and $\norm{\mathbf{f} - \pi(\mathbf{f})}$. We begin by proving that for any arbitrary\footnote{Certainly one could use the orthonormal basis $\big(\pmb{\phi}_k^\EL\big)_{1\leq k\leq 8}$ given by \eqref{eq:orthonormal basis EL}. For a sake of generality, we show here that the analysis can be done with an arbitrary orthonormal basis.} orthonormal basis $(\pmb{\psi}_k^\EL)_{1\leq k\leq 8}$ of $\Ker(\mathbf{L}^\EL)$ in $L^2(\R^3, \pmb{\mu}^{-1/2})$ we obtain the relation
\begin{equation} \label{eq:main relation}
\norm{\mathbf{f} - \pi(\mathbf{f})}_{\spaceVgamma}^2 \leq 2 \norm{\mathbf{f} - \pi^\EL(\mathbf{f})}_{\spaceVgamma}^2 + C_{\psi} \left( \norm{\pi^\EL(\mathbf{f})}_{\spaceV}^2 - \norm{\pi(\mathbf{f})}_{\spaceV}^2 \right),
\end{equation}
where $C_{\psi} = 16 \underset{1\leq k, \ell \leq 8}{\max} \left| \scalprod{\pmb{\psi}_k^\EL, \pmb{\psi}_\ell^\EL}_{\spaceVgamma} \right|$ and clearly each $\pmb{\psi}_k^\EL$ belongs to $\spaceVgamma$ since the elements of $\Ker(\mathbf{L}^\EL)$ take the form \eqref{eq:projection EL}. Indeed, we have
\begin{equation*}
\norm{\mathbf{f} - \pi(\mathbf{f})}_{\spaceVgamma}^2 \leq 2 \norm{\mathbf{f} - \pi^\EL(\mathbf{f})}_{\spaceVgamma}^2 + 2 \norm{\pi^\EL(\mathbf{f}) - \pi(\mathbf{f})}_{\spaceVgamma}^2,
\end{equation*}
and then for $\mathbf{g} = \pi^\EL(\mathbf{f}) - \pi(\mathbf{f}) \in \Ker(\mathbf{L}^\EL)$ we successively compute
\begin{equation*}
\begin{split}
\norm{\mathbf{g}}_{\spaceVgamma}^2 & = \scalprod{\sum_{k=1}^{8}\scalprod{\mathbf{g}, \pmb{\psi}_k^\EL}_{\spaceV}\pmb{\psi}_k^\EL,\sum_{\ell=1}^{8}\scalprod{\mathbf{g}, \pmb{\psi}_\ell^\EL}_{\spaceV}\pmb{\psi}_\ell^\EL}_{\spaceVgamma} \\[3mm]
%\sum_{i=1}^4 \int_{\R^3} \left| \sum_{k=1}^8 \scalprod{\mathbf{g}, \pmb{\psi}_k}_{\spaceV} \psi_{k,i} \right|^2 \scalprod{v}^\gamma \mu_i^{-1} \dd v \\[3mm]
& = \sum_{k, \ell = 1}^8 \scalprod{\mathbf{g}, \pmb{\psi}_k^\EL}_{\spaceV} \scalprod{\mathbf{g}, \pmb{\psi}_\ell^\EL}_{\spaceV} \scalprod{\pmb{\psi}_k^\EL, \pmb{\psi}_\ell^\EL}_{\spaceVgamma} \\[4mm]
& \leq \frac{1}{2} \underset{1\leq k, \ell \leq 8}{\max} \left| \scalprod{\pmb{\psi}_k^\EL, \pmb{\psi}_\ell^\EL}_{\spaceVgamma} \right| \sum_{k, \ell = 1}^8 \left( \scalprod{\mathbf{g}, \pmb{\psi}_k^\EL}_{\spaceV}^2 + \scalprod{\mathbf{g}, \pmb{\psi}_\ell^\EL}_{\spaceV}^2 \right) \\[4mm]
& \leq 8 \underset{1\leq k, \ell \leq 8}{\max} \left| \scalprod{\pmb{\psi}_k^\EL, \pmb{\psi}_\ell^\EL}_{\spaceVgamma} \right| \norm{\mathbf{g}}_{\spaceV}^2.
\end{split}
\end{equation*}
At last, since obviously $\Ker(\mathbf{L}^\EL + \mathbf{L}^\CH) \subset \Ker(\mathbf{L}^\EL)$, we have that $\pi^\EL (\pi) = \pi$ and therefore
\begin{equation*}
\begin{split}
\norm{\pi^\EL(\mathbf{f}) - \pi(\mathbf{f})}_{\spaceV}^2 & = \norm{\pi^\EL(\mathbf{f})}_{\spaceV}^2 - 2 \scalprod{\pi^\EL(\mathbf{f}), \pi(\mathbf{f})}_{\spaceV} + \norm{\pi(\mathbf{f})}_{\spaceV}^2 \\[3mm]
& = \norm{\pi^\EL(\mathbf{f})}_{\spaceV}^2 - 2 \scalprod{\mathbf{f}, \pi(\mathbf{f})}_{\spaceV} + \norm{\pi(\mathbf{f})}_{\spaceV}^2 \\[3mm]
& = \norm{\pi^\EL(\mathbf{f})}_{\spaceV}^2 - \norm{\pi(\mathbf{f})}_{\spaceV}^2.
\end{split}
\end{equation*}
This gives us estimate \eqref{eq:main relation}. The final aim is to prove that this last term can actually be estimated using the functional $\mathcal{E}(\pi^\EL(\mathbf{f}))$. To do this, we need an explicit expression for the $L^2$ norms of the projections $\pi^\EL$ and $\pi$, using their orthonormal representations in $L^2(\R^3, \pmb{\mu}^{-1/2})$. From the orthonormal basis \eqref{eq:orthonormal basis EL}, we first study the norm of $\pi^\EL(\mathbf{f})$. Recall the moment identities
\begin{equation*}
\int_{\R^3} \mu_i \dd v = c_{i,\infty}, \qquad \int_{\R^3} |v|^2 \mu_i \dd v = \frac{3 c_{i,\infty}}{m_i}, \qquad \int_{\R^3} |v|^4 \mu_i \dd v = \frac{15 c_{i,\infty}}{m_i^2},
\end{equation*}
and let us denote for simplicity the total concentration with $c_\infty = \sum_{i=1}^4 c_{i,\infty}$ and the total density with $ \rho_\infty = \sum_{i=1}^4 \rho_{i,\infty}$, where $\rho_{i,\infty} = m_i c_{i,\infty}$. Then, from Parseval's identity and the explicit expression \eqref{eq:projection EL} for the orthogonal projection onto $\Ker (\mathbf{L}^\EL)$, we get
\begin{equation*}
\begin{split}
\norm{\pi^\EL(\mathbf{f})}_{\spaceV}^2 & = \sum_{k=1}^8 \scalprod{\mathbf{f}, \pmb{\phi}^\EL_k}_{\spaceV}^2 = \sum_{k=1}^8 \scalprod{\pi^\EL(\mathbf{f}), \pmb{\phi}^\EL_k}_{\spaceV}^2 \\[4mm]
& = \sum_{i=1}^4 \left( \int_{\R^3} \frac{n_i + \frac{m_i}{2} e |v|^2}{\sqrt{c_{i,\infty}}} \mu_i \dd v \right)^2 + \ \sum_{\ell = 1}^3 \left( \sum_{i=1}^4 \int_{\R^3} \frac{m_i^2  u\cdot v}{\sqrt{\rho_\infty}} v_\ell \mu_i \dd v \right)^2 \\[3mm]
& \hspace{6cm} + \left( \sum_{i=1}^4 \int_{\R^3} \frac{n_i + \frac{m_i}{2} e |v|^2}{\sqrt{c_\infty}} \frac{m_i |v|^2 - 3}{\sqrt{6}} \mu_i \dd v \right)^2 \\[4mm]
& = \sum_{i=1}^4 c_{i,\infty} \left( n_i + \frac{3}{2} e \right)^2 + \rho_\infty |u|^2 + \frac{3}{2} c_\infty e^2.
\end{split}
\end{equation*}
Next, to compute the norm of $\pi(\mathbf{f})$ we need to introduce an orthonormal basis for $\Ker(\mathbf{L}^\EL + \mathbf{L}^\CH)$, starting from the seven linearly independent vectors \eqref{eq:basis CH}. We initially observe that $\pmb{\phi}_1$ and $\pmb{\phi}_2$ are orthogonal to one another and that the three vectors $\pmb{\phi}_{\ell + 3}$ related to momentum conservation are already orthogonal to all the others. The easier strategy is then to apply Gram--Schmidt procedure in this given order. We initially compute the orthogonal vectors
\begin{equation} \label{eq:orthonormal basis CH}
\begin{split}
\pmb{\phi}^\CH_1 = \left( \begin{array}{c} \mu_1\\[3mm] 0 \\[3mm] \mu_3 \\[3mm] 0 \end{array} \right), &\qquad \pmb{\phi}^\CH_2 = \left( \begin{array}{c} 0 \\[3mm] \mu_2 \\[3mm] 0 \\[3mm] \mu_4 \end{array} \right), \qquad \pmb{\phi}^\CH_3 = \left( \begin{array}{c} \displaystyle \frac{c_{3,\infty}}{c_{1,\infty}+c_{3,\infty}} \ \mu_1\\[6mm] \displaystyle -\frac{c_{4,\infty}}{c_{2,\infty}+c_{4,\infty}} \ \mu_2 \\[6mm] \displaystyle - \frac{c_{1,\infty}}{c_{1,\infty}+c_{3,\infty}} \ \mu_3 \\[6mm] \displaystyle \frac{c_{2,\infty}}{c_{2,\infty}+c_{4,\infty}} \ \mu_4 \end{array} \right), \\[8mm] 
\pmb{\phi}^\CH_{\ell + 3} = v_\ell \left( \begin{array}{c} m_1 \mu_1 \\[3mm] m_2 \mu_2 \\[3mm] m_3 \mu_3 \\[3mm] m_4 \mu_4 \end{array} \right), & \qquad \pmb{\phi}^\CH_7 = \left( \begin{array}{c} \displaystyle \frac{m_1 |v|^2 - 3}{2} \ \mu_1 \\[6mm] \displaystyle \frac{m_2 |v|^2 - 3}{2} \ \mu_2 \\[6mm] \displaystyle \frac{m_3 |v|^2 - 3}{2} \ \mu_3 \\[6mm] \displaystyle \frac{m_4 |v|^2 - 3}{2} \ \mu_4 \end{array} \right) + E_{12}^{34} \left( \begin{array}{c} \displaystyle - \frac{c_{3,\infty}}{c_{1,\infty}+c_{3,\infty}} \ \frac{c_{12}^{34}}{1+c_{12}^{34}} \ \mu_1 \\[6mm] \displaystyle - \frac{c_{4,\infty}}{c_{2,\infty}+c_{4,\infty}} \ \frac{1}{1+c_{12}^{34}} \ \mu_2 \\[6mm] \displaystyle \frac{c_{1,\infty}}{c_{1,\infty}+c_{3,\infty}} \ \frac{c_{12}^{34}}{1+c_{12}^{34}} \ \mu_3 \\[6mm] \displaystyle \frac{c_{2,\infty}}{c_{2,\infty}+c_{4,\infty}} \ \frac{1}{1+c_{12}^{34}} \ \mu_4 \end{array} \right),
\end{split}
\end{equation}
where $\ell \in \{ 1, 2, 3 \}$ and we have introduced the convenient notation $\displaystyle c_{12}^{34} = \frac{c_{1,\infty} + c_{3,\infty}}{c_{2,\infty} + c_{4,\infty}} \  \frac{c_{2,\infty} c_{4,\infty}}{c_{1,\infty} c_{3,\infty}}$. Orthonormalization follows from division by their respective norms, which write, for $\ell \in \{ 1, 2, 3 \}$,
\begin{gather*}
\norm{\pmb{\phi}^\CH_1}_{\spaceV}^2 = c_{1,\infty} + c_{3,\infty}, \qquad \norm{\pmb{\phi}^\CH_2}_{\spaceV}^2 = c_{2,\infty} + c_{4,\infty}, \\[8mm] 
\norm{\pmb{\phi}^\CH_3}_{\spaceV}^2 = \frac{c_{1,\infty} c_{3,\infty}}{c_{1,\infty} + c_{3,\infty}} \big( 1 + c_{12}^{34} \big) = \frac{c_{2,\infty} c_{4,\infty}}{c_{2,\infty} + c_{4,\infty}} \frac{1 + c_{12}^{34}}{c_{12}^{34}}, \qquad \norm{\pmb{\phi}^\CH_{\ell+3}}_{\spaceV}^2 = \rho_\infty, \\[10mm]
\norm{\pmb{\phi}^\CH_7}_{\spaceV}^2 = \frac{3}{2} c_\infty + \frac{c_{1,\infty} c_{3,\infty}}{c_{1,\infty} + c_{3,\infty}} \frac{c_{12}^{34}}{1 + c_{12}^{34}} \left(E_{12}^{34}\right)^2 = \frac{3}{2} c_\infty + \frac{c_{2,\infty} c_{4,\infty}}{c_{2,\infty} + c_{4,\infty}} \frac{1}{1 + c_{12}^{34}} \left(E_{12}^{34}\right)^2.
\end{gather*}
Using again Parseval's identity and noticing that $\Ker(\mathbf{L}^\EL)^\perp \subset \Ker(\mathbf{L}^\EL + \mathbf{L}^\CH)^\perp$, we can one more time resort to \eqref{eq:projection EL} and compute the norm of $\pi(\mathbf{f})$ as
\begin{equation*}
\begin{split}
& \norm{\pi(\mathbf{f})}_{\spaceV}^2 = \sum_{k=1}^7 \frac{\scalprod{\mathbf{f}, \pmb{\phi}^\CH_k}_{\spaceV}^2}{\norm{\pmb{\phi}^\CH_k}_{\spaceV}^2}  = \sum_{k=1}^7 \frac{\scalprod{\pi^\EL(\mathbf{f}), \pmb{\phi}^\CH_k}_{\spaceV}^2}{\norm{\pmb{\phi}^\CH_k}_{\spaceV}^2} = \\[5mm]
& = \frac{\left[ c_{1,\infty} \left( n_1 + \frac{3}{2} e \right) + c_{3,\infty} \left( n_3 + \frac{3}{2} e \right) \right]^2}{c_{1,\infty}+c_{3,\infty}} + \frac{\left[ c_{2,\infty} \left( n_2 + \frac{3}{2} e \right) + c_{4,\infty} \left( n_4 + \frac{3}{2} e \right) \right]^2}{c_{2,\infty}+c_{4,\infty}} \\[7mm]
& \quad + \frac{1}{\norm{\pmb{\phi}^\CH_3}_{\spaceV}^2} \left[ \frac{c_{1,\infty} c_{3,\infty}}{c_{1,\infty} + c_{3,\infty}} ( n_1 - n_3 ) + \frac{c_{2,\infty} c_{4,\infty}}{c_{2,\infty} + c_{4,\infty}} ( n_4 - n_2 ) \right]^2 + \rho_\infty |u|^2 \\[7mm]
& \quad + \frac{1}{\norm{\pmb{\phi}^\CH_7}_{\spaceV}^2} \left[ \frac{3}{2} c_\infty e - \frac{c_{1,\infty} c_{3,\infty}}{c_{1,\infty} + c_{3,\infty}} \frac{c_{12}^{34}}{1 + c_{12}^{34}}  ( n_1 - n_3 ) E_{12}^{34} + \frac{c_{2,\infty} c_{4,\infty}}{c_{2,\infty} + c_{4,\infty}} \frac{1}{1 + c_{12}^{34}} ( n_4 - n_2 ) E_{12}^{34} \right]^2.
\end{split}
\end{equation*}
Recalling the definition of $c_{12}^{34}$, from the relation $\frac{c_{1,\infty} c_{3,\infty}}{c_{1,\infty} + c_{3,\infty}} c_{12}^{34} = \frac{c_{2,\infty} c_{4,\infty}}{c_{2,\infty} + c_{4,\infty}}$ we deduce that the last numerator can then be recast as
\begin{equation*}
\left[ \frac{3}{2} c_\infty e + \frac{c_{1,\infty} c_{3,\infty}}{c_{1,\infty} + c_{3,\infty}} \frac{c_{12}^{34}}{1 + c_{12}^{34}}  ( n_4 + n_3 - n_2 - n_1 ) E_{12}^{34} \right]^2.
\end{equation*}
We can now study the difference between the two orthogonal projections and successively compute
\begin{equation*}
\begin{split}
& \norm{\pi^\EL(\mathbf{f})}_{\spaceV}^2 - \norm{\pi(\mathbf{f})}_{\spaceV}^2 = \frac{c_{1,\infty} c_{3,\infty}}{c_{1,\infty}+c_{3,\infty}} ( n_3 - n_1 )^2 + \frac{c_{2,\infty} c_{4,\infty}}{c_{2,\infty}+c_{4,\infty}} ( n_4 - n_2 )^2 \\[5mm]
& \hspace*{5cm} - \frac{c_{2,\infty} c_{4,\infty}}{c_{2,\infty}+c_{4,\infty}} \frac{c_{12}^{34}}{1+c_{12}^{34}} (n_4 - n_2)^2 - \frac{c_{1,\infty} c_{3,\infty}}{c_{1,\infty}+c_{3,\infty}} \frac{1}{1+c_{12}^{34}} ( n_3 - n_1 )^2 \\[5mm]
& \hspace*{5cm} + \frac{2}{\norm{\pmb{\phi}^\CH_3}_{\spaceV}^2} \frac{c_{1,\infty} c_{3,\infty}}{c_{1,\infty}+c_{3,\infty}} \frac{c_{2,\infty} c_{4,\infty}}{c_{2,\infty}+c_{4,\infty}} ( n_3 - n_1 )( n_4 - n_2 ) \\[5mm]
& \hspace*{5cm} + \frac{3}{2} c_\infty e^2 - \frac{\left[ \frac{3}{2} c_\infty e + \frac{c_{1,\infty} c_{3,\infty}}{c_{1,\infty} + c_{3,\infty}} \frac{c_{12}^{34}}{1 + c_{12}^{34}}  ( n_4 + n_3 - n_2 - n_1 ) E_{12}^{34} \right]^2}{\frac{3}{2} c_\infty + \frac{c_{1,\infty} c_{3,\infty}}{c_{1,\infty} + c_{3,\infty}} \frac{c_{12}^{34}}{1 + c_{12}^{34}} \left(E_{12}^{34}\right)^2}
\end{split}
\end{equation*}
\begin{equation*}
\begin{split}
& = \frac{c_{1,\infty} c_{3,\infty}}{c_{1,\infty}+c_{3,\infty}} \frac{c_{12}^{34}}{1+c_{12}^{34}} (n_4 + n_3 - n_2 - n_1)^2 + \frac{3}{2} c_\infty e^2 \\[3mm]
& \hspace*{5cm} - \frac{\left[ \frac{3}{2} c_\infty e + \frac{c_{1,\infty} c_{3,\infty}}{c_{1,\infty} + c_{3,\infty}} \frac{c_{12}^{34}}{1 + c_{12}^{34}}  ( n_4 + n_3 - n_2 - n_1 ) E_{12}^{34} \right]^2}{\frac{3}{2} c_\infty + \frac{c_{1,\infty} c_{3,\infty}}{c_{1,\infty} + c_{3,\infty}} \frac{c_{12}^{34}}{1 + c_{12}^{34}} \left(E_{12}^{34}\right)^2} \\[5mm]
& = \frac{\frac{3}{2} c_\infty \frac{c_{1,\infty} c_{3,\infty}}{c_{1,\infty} + c_{3,\infty}} \frac{c_{12}^{34}}{1 + c_{12}^{34}}}{\frac{3}{2} c_\infty + \frac{c_{1,\infty} c_{3,\infty}}{c_{1,\infty} + c_{3,\infty}} \frac{c_{12}^{34}}{1 + c_{12}^{34}} \left(E_{12}^{34}\right)^2} \left( n_4 + n_3 - n_2 - n_1 - e E_{12}^{34} \right)^2 \\[7mm]
& \leq \frac{c_{1,\infty} c_{3,\infty}}{c_{1,\infty} + c_{3,\infty}} \frac{c_{12}^{34}}{1 + c_{12}^{34}}\mathcal{E}(\pi^\EL(\mathbf{f})) \\[7mm]
& \leq c_{\infty}^4\mathcal{E}(\pi^\EL(\mathbf{f})),
\end{split}
\end{equation*}
where we have used the definition of $c_{12}^{34}$ to determine the last estimate. Going back to the previous relation \eqref{eq:main relation}, we finally infer that
% \begin{equation*}
% \norm{\mathbf{f} - \pi(\mathbf{f})}_{\spaceVgamma}^2 \leq 2 \norm{\mathbf{f} - \pi^\EL(\mathbf{f})}_{\spaceVgamma}^2 + \kappa^\prime \mathcal{E}(\pi^\EL(\mathbf{f})),
% \end{equation*}
\begin{equation} \label{eq:C_psi}
    - \norm{\mathbf{f} - \pi^\EL(\mathbf{f})}_{\spaceVgamma}^2 \leq - \frac 12\norm{\mathbf{f} - \pi(\mathbf{f})}_{\spaceVgamma}^2  + \frac{C_{\psi}c_{\infty}^4}{2} \mathcal{E}(\pi^\EL(\mathbf{f})).
\end{equation}
%where we have defined
%\begin{equation*}
%\kappa^\prime = \kappa \frac{\frac{3}{2} c_\infty \frac{c_{1,\infty} c_{3,\infty}}{c_{1,\infty} + c_{3,\infty}} \frac{c_{12}^{34}}{1 + c_{12}^{34}}}{\frac{3}{2} c_\infty + \frac{c_{1,\infty} c_{3,\infty}}{c_{1,\infty} + c_{3,\infty}} \frac{c_{12}^{34}}{1 + c_{12}^{34}} \left(E_{12}^{34}\right)^2}.
%\end{equation*}
%Performing straightforward algebraic manipulations on the fraction $\frac{c_{12}^{34}}{1 + c_{12}^{34}}$, we can then determine a very crude estimate of the constant $\kappa^\prime$ in terms of the total concentration $c_{\infty}$. In particular, we recover that $\kappa^\prime \leq \kappa c_{\infty}^4$.

%%%%%%%%%%%%%%%    PROOF OF MAIN RESULT - STEP 5    %%%%%%%%%%%%%%%%
\medskip
\noindent \textbf{Step 5 -- Combining the estimates to conclude.} At last, it is clear that injecting estimate \eqref{eq:C_psi} into \eqref{eq:main estimate 2} for $\delta \in (0,1]$,
% the upper bound
% \begin{equation*}
% - \mathcal{E}(\pi^\EL(\mathbf{f})) \leq C^\prime \left( 2 \norm{\mathbf{f} - \pi^\EL(\mathbf{f})}_{\spaceVgamma}^2 - \norm{\mathbf{f} - \pi(\mathbf{f})}_{\spaceVgamma}^2 \right)
% \end{equation*}
% in the main estimate \eqref{eq:main estimate 2}, 
we can infer that the linearized reactive operator $\mathbf{L} = \mathbf{L}^\EL + \mathbf{L}^\CH$ satisfies the upper bound
\begin{equation*}
\begin{split}
\scalprod{\mathbf{L}(\mathbf{f}), \mathbf{f}}_{\spaceV} & \leq -\frac{\delta}{2}(\lambda_\EL-16C_{\nu} \eta) \norm{\mathbf{f} - \pi(\mathbf{f})}_{\spaceVgamma}^2 \\[5mm]
& \qquad - \frac{1}{2} \left( \eta C_b - \delta(\lambda_\EL - 16 C_{\nu}\eta) C_{\psi} c_{\infty}^4 \right) \mathcal{E}(\pi^\EL(\mathbf{f})),
\end{split}
\end{equation*}
where $\eta \in (0,1]$ must verify $\eta < \lambda_\EL/(16 C_{\nu})$. Thus, by choosing a sufficiently small $\delta\in (0,1]$, we obtain the desired spectral gap estimate
\begin{equation*}
    \scalprod{\mathbf{L}(\mathbf{f}), \mathbf{f}}_{\spaceV} \leq - \lambda \norm{\mathbf{f} - \pi(\mathbf{f})}_{\spaceVgamma}^2,
\end{equation*}
with $\lambda = \delta(\lambda_\EL - 16 C_{\nu}\eta)/2$. In order to optimize $\lambda$, we choose $\delta = \min\left\{1,\ \frac{\eta C_b}{(\lambda_\EL - 16C_{\nu}\eta)C_{\psi}c_{\infty}^4}\right\}$, which in turn leads to $\lambda = \min\left\{\frac{\lambda_\EL-16C_{\nu}\eta}{2},\ \frac{\eta C_b}{2C_{\psi}c_{\infty}^4}\right\}$. Therefore, if $\frac{\lambda_\EL C_{\psi}c_{\infty}^4}{C_b+16C_{\nu}C_{\psi}c_{\infty}^4} \leq 1$, by taking $\eta= \frac{\lambda_\EL C_{\psi}c_{\infty}^4}{C_b+16C_{\nu}C_{\psi}c_{\infty}^4}$ we obtain an explicit expression for the spectral gap, that is given by
\begin{equation*}
\lambda = \frac{\lambda_\EL C_b}{2( C_b +  16C_{\nu}C_{\psi}c_{\infty}^4)},
\end{equation*}
where the constants
\begin{equation*}
C_{\nu} = \max_{1 \leq i \leq 4} \overline{\nu}_i, \qquad C_b = \underline{\nu}_1 \int_{\R^3} (1 + |v|^\gamma) \mu_1(v) \dd v, \qquad C_{\psi} = \underset{1\leq k, \ell \leq 8}{\max} \left| \scalprod{\pmb{\psi}_k^\EL, \pmb{\psi}_\ell^\EL}_{\spaceVgamma} \right|,
\end{equation*}
come respectively from estimates \eqref{eq:C_nu}, \eqref{eq:C_b} and \eqref{eq:C_psi}, and we recall that $(\pmb{\psi}^\EL_k)_{1\leq k\leq 8}$ is an arbitrary orthonormal basis of $\Ker(\mathbf{L}^\EL)$ in $L^2(\R^3, \pmb{\mu}^{-1/2})$. This completes the proof of Theorem \ref{theorem:reactive SG}.

\section{Conclusion}

In this work, we proved the existence of an explicit spectral gap for the linearized Boltzmann operator modeling reactive gaseous mixtures that involve reversible bimolecular chemical reactions of the form $\mathcal{S}_1 + \mathcal{S}_2 \leftrightharpoons \mathcal{S}_3 + \mathcal{S}_4$. This linearized operator appears when considering solutions to the reactive Boltzmann system, perturbed around the global Maxwellian states \eqref{eq:Maxwellian}. The study of quantitative spectral gap estimates for such operators in the case of elastic collisions inside monoatomic gases consisting of one or more species has been extensively investigated in the literature, while the analysis of linearized operators for reactive mixtures was only limited to non-constructive results that could not provide information on the magnitude of their spectral gap. Our work filled this gap by obtaining explicit estimates on the spectral gap of the linearized reactive Boltzmann operator, in terms of the total concentration of the mixture, the different molecular masses of the species, the chemical binding energy of the reactions, the elastic and reactive collision kernels. Our proof is inspired by the ideas developed in \cite{BriDau,DauJunMouZam} where the spectral gap of the elastic monospecies operator is used to control that of its multi-species counterpart. Similarly, we split here the linearized reactive operator $\mathbf{L} = \mathbf{L}^\EL + \mathbf{L}^\CH$ into an elastic and a chemical part. The former $\mathbf{L}^\EL$ is known to possess a spectral gap from the results in \cite{BriDau,DauJunMouZam} and we showed how to exploit this information in order to control $\mathbf{L}^\CH$, by quantifying in particular the relaxation of equilibrium distributions $\pi^\EL(\mathbf{f}) \in \Ker(\mathbf{L}^\EL)$ for the linearized elastic operator $\mathbf{L}^\EL$ toward equilibrium states $\pi(\mathbf{f}) \in \Ker(\mathbf{L}^{\EL} + \mathbf{L}^{\CH})$ for the full linearized reactive operator $\mathbf{L}$.

\smallskip
The result of this work provides the basis for the development of a quantitative Cauchy theory of close-to-equilibrium solutions to reactive Boltzmann systems and the study of their convergence to equilibrium. It is also a first step toward the rigorous derivation of hydrodynamic limits for these kinetic models, leading to reaction--diffusion systems that involve nonlinear mass action kinetics \cite{BisDes} in the spirit of the recent results \cite{BonBri1,BonBri2,BriGre} on the derivation of the multi-species Maxwell--Stefan and Fick equations. An ongoing work by the first author is addressing this particular problem. Another interesting open question is the study of relaxation to equilibrium for solutions to the reactive Boltzmann system in a global setting, i.e. far from equilibrium. Further perspectives include the investigation of fast reaction hydrodynamic limits, involving more complicate chemical kinetics like the celebrated Michaelis--Menten kinetics, and the derivation of macroscopic models that display both cross-diffusion and reaction dynamics \cite{AnwBisSalSoa,AnwGonSoa}.

%%%%%%%%%%%%%%%%%%%%%%%%%%%%%%%%%%%%%%%%%%%%%%%%%%%%%%%%%%%%
%%%%%%%%%%%%%%%%%%%%%%%%%%%%%%%%%%%%%%%%%%%%%%%%%%%%%%%%%%%%
%%%%%%%%%%%%%%%%%     AKNOWLEDGMENTS    %%%%%%%%%%%%%%%%%%%
%%%%%%%%%%%%%%%%%%%%%%%%%%%%%%%%%%%%%%%%%%%%%%%%%%%%%%%%%%%%
%%%%%%%%%%%%%%%%%%%%%%%%%%%%%%%%%%%%%%%%%%%%%%%%%%%%%%%%%%%%

\bigskip
\noindent \textbf{Acknowledgments.} We wish to thank the two anonymous referees that provided helpful suggestions to improve this work. AB acknowledges the support from the Italian National Group of Mathematical Physics (GNFM--INdAM), from the Austrian Science Fund (FWF) through the Lise Meitner project No. M-3007 (Asymptotic Derivation of Diffusion Models for Mixtures) and from the European Union’s Horizon Europe research and innovation programme, under the Marie Skłodowska-Curie grant agreement No. 101110920, project MesoCroMo (A Mesoscopic approach to Cross-diffusion Modelling in population dynamics). BQT received funding from the FWF project number I-5213 (Quasi-steady-state approximation for PDE). Both authors acknowledge support from the COST Action CA18232 MAT-DYN-NET, funded by COST (European Cooperation in Science and Technology).

%%%%%%%%%%%%%%%%%%%%%%%%%%%%%
%%%%%%%%%%%%%%%%%%%%%%%%%%%%%
\begin{figure}[h!]
\begin{flushleft}
\includegraphics[scale=0.3]{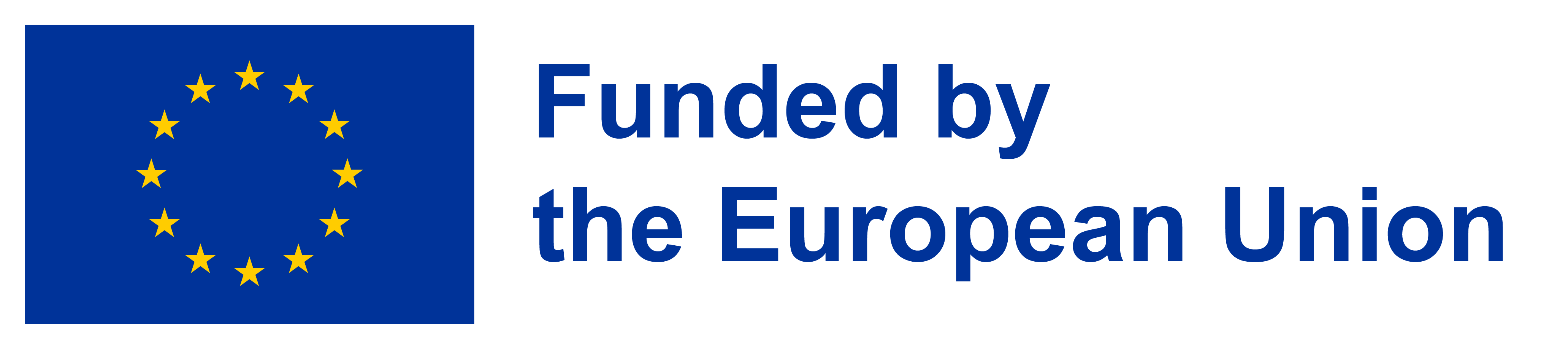}
\end{flushleft}
\end{figure}
%%%%%%%%%%%%%%%%%%%%%%%%%%%%%
%%%%%%%%%%%%%%%%%%%%%%%%%%%%%

\bigskip
\noindent \textbf{Disclaimer.} Funded by the European Union. Views and opinions expressed are however those of the author(s) only and do not necessarily reflect those of the European Union or of the European Research Executive Agency (REA). Neither the European Union nor the granting authority can be held responsible for them.

%%%%%%%%%%%%%%%%%%%%%%%%%%%%%%%%%%%%%%%%%%%%%%%%%%%%%%%%%%%%
%%%%%%%%%%%%%%%%%%%%%%%%%%%%%%%%%%%%%%%%%%%%%%%%%%%%%%%%%%%%
%%%%%%%%%%%%%%%%%%%%    BIBLIOGRAPHY   %%%%%%%%%%%%%%%%%%%%%
%%%%%%%%%%%%%%%%%%%%%%%%%%%%%%%%%%%%%%%%%%%%%%%%%%%%%%%%%%%%
%%%%%%%%%%%%%%%%%%%%%%%%%%%%%%%%%%%%%%%%%%%%%%%%%%%%%%%%%%%%

\bigskip
%\nocite{*}
\bibliographystyle{plain}
\bibliography{Bibliography_SGRBE}

\setlength\parindent{0pt}

\end{document}